\documentclass[runningheads]{llncs}
\usepackage{amsmath, amssymb}
\usepackage{graphicx} 
\usepackage{amsfonts}
\usepackage{xcolor}
\usepackage{hyperref}
\usepackage[utf8]{inputenc}
\usepackage{subcaption}
\usepackage[boxruled]{algorithm2e}
\usepackage{booktabs}
\usepackage{cite}
\usepackage[title]{appendix}
\usepackage[T1]{fontenc}
\SetCommentSty{commfont}
\SetKwComment{Comm}{$\rhd\ $}{}
\newcommand{\N}{\mathcal{N}}
\newcommand{\I}{\mathbf{I}}
\newcommand{\E}{\mathbb{E}}
\newcommand{\diff}{\mathrm{d}}

\DeclareMathOperator*{\argmax}{argmax}
\newtheorem{assumption}[theorem]{Assumption}
\newtheorem{cor}[theorem]{Corollary}
\newtheorem{lem}[theorem]{Lemma}

\usepackage{color}

\begin{document}
\title{Global convergence analysis of mixtures of Exponential densities}
%
%
\author{Rajita Chandak\inst{1} \and
Kathryn Dullerud\inst{1}}
\authorrunning{Chandak and Dullerud}
%
\institute{
  Institute of Mathematics, \'{E}cole Polytechnique F\'{e}d\'{e}rale de Lausanne
  (EPFL), Switzerland
  \email{rajita.chandak@epfl.ch}
}
\maketitle              
\begin{abstract}
  The theoretical foundations of the EM algorithm are often thought of
  in the context of Gaussian mixture models, However, the practical use cases of
  the EM algorithm span beyond Gaussian models. This paper establishes the first
  step towards understanding the behavior of the EM algorithm under mixtures of
  non-Gaussian densities. We show that a mixture of two Exponential
  distributions can be approximated by the EM algorithm at the sub-Exponential
  rate of convergence in at most $\log (n)$ iterations. The results here show
  that extending away from Gaussian mixture models does not affect the statistical
  performance of the EM algorithm. Furthermore, we present generalizations of
  typical assumptions in the Gaussian setting like minimum mean-separation and
  signal-to-noise ratio to the sub-Exponential setting.
  A simulation study is used to highlight the empirical performance of EM for
  mixtures of exponentials with promising results for the extension of
  existing theory to a larger class of mixture models.
  \keywords{EM Algorithm \and Mixture models \and Convergence rates}
\end{abstract}

\section{Introduction}\label{intro}

Maximum likelihood estimation is often considered a statistician's bread and
butter for parametric estimation due to its many desirable properties, such as
asymptotic efficiency and consistency. However, direct computation of the
maximum likelihood estimates is not always tractable, even for many cases of
parametric models. One particularly popular case of intractability is mixtures
of parametric densities. Mixture models are highly flexible structures
that are used in a number of statistical settings, from gene expression
analysis~\cite{ghosh2002mixture}
to natural language processing~\cite{blei2003latent}, that importantly allow
for the modeling of data heterogeneity~\cite{mclachlan2019finite}. As such,
there is great interest in methods or heuristics that attempt to apply the
likelihood principle in parametric estimation for these models. One of the most
popular approaches to such problems is the Expectation-Maximization Algorithm
(EM), introduced formally by~\cite{dempster1977maximum}, but the form of which
was used in specific settings prior to their publication, notably
in~\cite{baum1970maximization, sundberg1974maximum}.

The EM Algorithm owes its success to its relatively simple algorithmic
formulation and to the fact that it has been shown to produce good results in
practice \cite{redner1984mixture, daskalakis2017ten}.
However, up until quite recently, little was understood about its global
statistical guarantees, even in specialized cases. Most recent work, as a
result, has focused on Gaussian mixture models (GMMs)
\cite[and references therein]{wu2019randomly,
  zhao2020statistical, dwivedi2020singularity} and
mixtures of linear regressions with Gaussian errors (MLRs) \cite{kwon2019global,
  kwon2020converges, tao2024convergence}, many with further simplifications like
equally-weighted two-component mixtures. The
global convergence properties of other types of mixture models, however, remain
largely unexplored, especially, under the EM umbrella.

This paper provides insight into the applicability of EM to mixtures of two
exponential densities with theoretical guarantees on the rate of convergence,
that we conjecture is near-minimax optimal for mixtures of sub-Exponential
distributions. Furthermore, we investigate the
limiting regimes of exponential mixtures (analogous to the degenerate setting in
GMMs) that provide insight into the analogous notions of ``well-defined'' exponential mixtures
and SNR type conditions that are typically found in the GMM literature.
The goal of this paper is to provide a foundational basis of establishing the
theory of mixture models to non-Gaussian cases and thus extending the
applicability of iterative algorithms like EM to a larger class of models.

In order to inform the specifications of applying the EM algorithm to mixtures
of exponential densities, we first highlight existing results in the literature.
Here we summarize relevant literature, emphasizing primarily the two key aspects of
our objective (1) mixture models and (2) EM algorithm with an emphasis on
recent advances within the fields that are particularly relevant to the present work.

\subsubsection{Mixture models}
Mixture models are important probabilistic tools for modeling in a
variety of settings, due to their ability to flexibly model complex
distributions~\cite{mclachlan2019finite, lindsay1995mixture}. In particular, a
mixture model is said
to represent a population composed of distinct sub-populations, the
identification of which is generally unknown for observed data.
This work is related to literature concerning parameter estimation for finite
mixture models. Though not the focus of the present work, there have been many
different approaches to resolving parameter estimation of parametric mixture
models, largely focused on the mixture of Gaussians. These include method of
moments~\cite{anandkumar2012method}, which uses the notion of matching moments
to identify the parameters in terms of the moment generating function and
spectral methods~\cite{vempala2004spectral, kannan2005spectral}, which use
principal components of the singular-value decomposition (SVD) of the data to
identify (non-overlapping) mixture components.
Most commonly, however, is the use of the EM algorithm to identify mixture
models, for which there is a plethora of applied work as well as a recent surge
in theoretical work, as described in further detail below.

\subsubsection{EM algorithm}
The EM algorithm was formalized by~\cite{dempster1977maximum}, where it
was introduced with specific applications to mixtures of Gaussian densities and
linear regressions. The first theoretical guarantees, however, were provided by
~\cite{wu1983convergence} who proved the algorithm, applied to mixtures of
curved exponential family densities always converges to stationary points of the
likelihood under mild conditions. Furthermore, when applied to mixtures of
unimodal densities with sufficient smoothness, the algorithm finds the MLE.~\cite{redner1984mixture} proved local convergence, without precise rates, to
the MLE for exponential family mixture
models.~\cite{balakrishnan2017statistical} provide non-asymptotic local
convergence guarantees for EM-type algorithms under smoothness conditions, that
can be shown to apply to Gaussian mixtures.

More work has focused on specific simplifications of the Gaussian mixture
model (GMM) to obtain stronger guarantees or more precises statements on
consistency, including rates of convergence.~\cite{xu2016global, daskalakis2017ten} both showed convergence of the
population EM iterates to the true parameter in the balanced (equally-weighted)
2-component Gaussian mixture case. Further,~\cite{xu2016global} show that, under
the same balanced 2-GMM, in the double limit as both sample size and number of
iterations go to infinity, the difference between the population and
finite-sample iterates disappear.
Additionally,~\cite{daskalakis2017ten} prove that in the finite-sample case the
iterates will converge with high probability under a sample-splitting scheme and with a
warm-start after random initialization, provided sufficient number of samples for each
iterate.~\cite{wu2019randomly} improve on these
by proving a near minimax-optimal rate for the 2-component Gaussian case with
random initialization.

Recently, work has considered the EM Algorithm for GMMs with either more than 2
components or unequal mixing weights~\cite{zhao2020statistical}, or even a
mispecified number of components~\cite{dwivedi2020singularity}. Some
impossibility results on global convergence for more than 2-component
GMM have also been identified~\cite{jin2016local}, showcasing that the application
of EM is neither straightforward nor universal.

Another recent line of interest concerning EM is with applications to mixtures of
linear regressions (MLRs) which can, in some cases when conditioning on
covariates, be written as an alternative formulation of the classification
problem \cite{viele2002modeling, tao2024convergence} and in other cases can be
of independent interest. Recent work, similar to the classification case has
focused on Gaussian error linear regression models, with a larger emphasis on
2-component mixtures. In particular, ~\cite{kwon2019global} prove global convergence of
the EM for a mixture of two linear regressions, while~\cite{kwon2020converges}
show local convergence for an arbitrary number of components for MLRs under
constraints on the signal-to-noise ratio.

\subsection{Outline}
Despite this recent flurry of work on convergence of EM in mixture model
settings, there is almost no literature on non-Gaussian mixture models, which is pertinent in a number of domains.
As a result we focus the present work on the convergence of the EM Algorithm for
balanced 2-component mixtures of exponential densities to provide a foundation
for extending theory to general mixtures of exponential family densities.

The rest of this paper is organized as follows: Section \ref{sec:prelims} covers the
necessary definitions and provides an overview of the EM Algorithm. Section
\ref{sec:results} covers both the population and finite-sample convergence
results. Section \ref{sec:sims} provides simulations that help illustrate the
theoretical results, test the tightness of the assumptions and potential
extensions of this work to general cases that could motivate future work.
Section \ref{sec:conclusion} concludes and discusses some future directions of
interest. Appendices~\ref{app:main_proofs} and~\ref{app:technical} contain all
the proofs and related technical results.

\section{Preliminaries \& Problem Set-Up}\label{sec:prelims}
We begin this section with a thorough treatment of the notation used throughout
the paper. We let capitalized letters, such as $X \in \mathbb{R}$, refer to
random variables.  Realizations of these random
variables will be given by the corresponding lower-case letters. We
let $p_\theta(\cdot)$ denote the probability density function of a continuous
random variable with parameter $\theta$. $X \sim p_\theta$ is used to denote
that the random variable $X$ is distributed according to the law of $p_\theta$.
Further, $\mathcal{L}(\cdot \hspace{1mm}; \theta)$ refers to the likelihood
function of a sample of a continuous random variable, and $\ell(\cdot
\hspace{1mm} ; \theta)$, the log-likelihood. Moreover, in a slight abuse of
notation, we write $\mathbb{E}_\theta[\cdot]$ to refer to the expectation of a
random variable with respect to its distribution parameterized by $\theta$.
We use $a_n \lesssim b_n$ to mean $a_n \leq cb_n$ for some universal constant
$c>0$, similarly $a_n \gtrsim b_n$ implies $a_n \geq cb_n$ for some universal
constant $c>0$.

We now turn to describing the principal data generation model.

\subsection{Finite Mixture Models}\label{sec:fmms}
Formally, we say a random variable $X$ comes from a mixture of distributions,
when its density can be written as:
\[
p_{\mathbb{\theta}}(x) = \sum_{k=1}^K \pi_k p_{\theta_k}(x),
\]
where $K$ is the number of components of the mixture, $\mathbb{\theta}
= [\theta_1,...,\theta_K]$ is the vector of the parameters that define the
distribution of each component. Let
$p_{\theta_k}$ be the density of the $k$th component of the mixture,
and $\pi_k$ the probability that $X$ comes from the $k$th component, sometimes
referred to as the mixing proportion. To ensure that we are generating data from
a valid distribution, we require $0 \leq \pi_k \leq 1$ for each $k$, with strict
inequality for a well-defined $K$-component mixture, and $\sum_{k=1}^K \pi_k =
1$. In practice, the mixing proportions $(\pi_1,...,\pi_K)$ are unknown and can
be treated as an additional vector of parameters that must be estimated to fully specify the
data-generating process. Here we assume $K < \infty$ to be a known constant.
Although not necessary, typically it is assumed that each of the
$K$ components come from the same parametric distribution family. As such, it is
standard language when refering to a mixture of Gaussians to mean that each
$p_{\theta_k}$ is Gaussian.

In order to facilitate mathematical analysis, mixture models are often written as latent
variable models. In particular, we may represent the component from which a
realization of the random variable $X$ comes as a random variable itself, say
$Z$. In the finite mixture case, $Z \in \{1,...,K\}$ is a discrete random
variable such that $\mathbb{P}(Z=k) = \pi_k$. Then the marginal density of $X$
can be rewritten in the following manner:
\begin{align*}
p_{\mathbb{\theta}}(x) = \sum_{k=1}^K \pi_k p_{\theta_k}(x|Z=k).
\end{align*}
A typical approach to estimating the unknown parameters $\{\theta_k,
\pi_k\}_{k=1}^K$ is to generate the maximum likelihood estimates (MLEs), which
one would hope has the desirable consistency properties.
However, with a mixture model it is not possible to directly maximize the
likelihood, due to the latent variables $Z$ that are unobserved in the data.
In particular, the likelihood, using the latent variable representation of mixture
models, takes on the form:
\begin{align*}
\mathcal{L}(X_1,..,X_n; \theta) = \prod_{i=1}^n\sum_{k=1}^K \pi_k p_{\theta_k}(X_i|Z_i=k).
\end{align*}
The EM algorithm is used precisely to approximate the maximizer of the likelihood
in an iterative manner that can be thought of as a gradient-descent approach to
iteratively get closer to the stationary point of the likelihood function,
that we hope to identify as the true maximizer.

\subsection{The EM Algorithm}
\label{sec:em}
The principle of the EM algorithm is to use the hypothetically simple likelihood
of the ``complete'' data in order to make estimates based only on the observed
data. We simplify our presentation of the EM algorithm with respect to the
latent variable formulation. Additional detailed explanations around the
construction of the EM Algorithm can be found
in~\cite{wu1983convergence,mclachlan2008algorithm, davison2008statistical}.

Let $X \in \mathcal{X}$ denote the observed data, and $Z \in \mathcal{Z}$, the
unobserved data. We assume that $(X,Z)$ is generated from a parameterized joint
distribution, denoted $p_\theta(X,Z)$. Let $X_1,...,X_n$ be an i.i.d.\ sample of the
observed data.
Since the $Z_1, \ldots, Z_n$ are never observed, the log-likelihood is
simplified by taking expectations of the latent variable.
\begin{align*}
  \mathbb{E}_{Z\sim p_{\theta'}(\cdot|X)}[\ell(X,Z; \theta)]
  = \ell(X; \theta) +
  \mathbb{E}_{Z\sim p_{\theta'}(\cdot|X)}[\ell(Z|X; \theta)].
\end{align*}
This is the expectation step of the EM algorithm. It requires some initial guess
of the parameters, $\theta'$.
Typically this expectation is denoted as a $Q$ function with respect to the
current estimate of the parameters, $\theta'$. That is,
\begin{align*}
Q(\theta|\theta') = \ell(X; \theta) + C(\theta|\theta').
\end{align*}
The natural next step is to maximize this estimated log-likelihood function.
This is precisely what is referred to as the `M-step' of the EM algorithm.
One can verify, by application of Jensen's inequality and leveraging the
concavity of the log function, that if there exists some $\tilde{\theta}$ such
that $Q(\tilde{\theta}|\theta') \geq Q(\theta'|\theta')$, then $\ell(X;
\tilde{\theta}) \geq \ell(X;\theta')$, conditional on $\theta'$.
Note that the implication does not hold in the reverse direction. That is
$\ell(X; \tilde{\theta}) \geq \ell(X;\theta') \nRightarrow
Q(\tilde{\theta}|\theta') \geq Q(\theta'|\theta')$. As a result, it is logical
to maximize $Q$ with respect
to $\theta$ in order to approximate the true maximizer of the likelihood
$\ell(X;\theta)$. This illustrates that the EM algorithm is based on the notion
of never taking a step in the ``wrong direction'', similar to the idea of
gradient descent. In fact, the $C(\cdot|\theta')$ can be thought of as a step in
gradient descent of a fixed step size, as identified
by~\cite{jamshidian1993conjugate}.

The E and M steps are then iterated over to generate a sequence of estimators
$\hat{\theta}$ that converge to a stationary point of $Q$ with respect to some
loss metric (see Algorithm~\ref{alg:EM} for pseudo-code), provided the
log-likelihood is bounded~\cite{wu1983convergence}.
\begin{algorithm}[h!]
  \caption{The Expectation-Maximization algorithm}~\label{alg:EM}
  Given data $X_1, \ldots, X_n$,
  set $t=0$, initialize $\theta_{0}$, and choose
  tolerance level $r_{\text{tol}}$ and loss function $L(\cdot, \cdot)$.
  \\
  \Repeat{$L(\theta_{t}, \theta_{t-1}) < r_{\text{tol}}$}{
    Compute $Q(\theta|\theta_{t-1})$ conditional on the sample $x_1, \ldots, x_n$ \hspace{0.2in}\tcp{E-step}
      $\theta_{t+1} = \argmax_{\theta\in\Theta} Q(\theta|\theta_{t})$ \hspace{1.8in}\tcp{M-step}
    Set $t = t + 1$\hspace{1.9in}\tcp{Update iteration count}
  }
\end{algorithm}

The main question now is whether the limit point of the sequence of estimators
approximates the true parameter consistently in a finite number of iterations under
minimal assumptions on the data generating process.~\cite{wu1983convergence} can be credited with the first successful attempt to
characterize the convergence of the EM Algorithm. The authors prove that the EM
algorithm always converges to a stationary point of $\ell(X;\theta)$
whenever $Q(\theta|\theta')$ is continuous in both $\theta$ and $\theta'$ and
further, that this condition is easily satisfied for mixtures of curved
exponential family distributions.

Though the E- and M-steps are treated
separately in the general formulation, one attractive property of the EM is the
existence of closed-form solutions for finite mixtures of
Gaussians~\cite{mclachlan2008algorithm, wu2019randomly}.
Thus, most of the existing literature exclusively deal with Gaussian mixture
models. In particular,~\cite{xu2016global, daskalakis2017ten} show that for the
symmetric 2-component GMM, formulated as
\begin{align*}
  p_{\theta} = \frac{1}{2}\N_d(\theta, \I_d) + \frac{1}{2}\N_d(-\theta, \I_d)
\end{align*}
where $\N_d$ represents a $d-$dimensional Gaussian distribution and $\I_d$ is a
$d\times d$ identity matrix, the population EM iterates $\theta_t$ converge to
the true parameters at a geometric rate.
Furthermore, these results show that an increasing number of iterations are required
to achieve convergence of the algorithm, although the dependency of the number
of iterations on $n$ varies
between $\log(n)$ and $\sqrt{n}$, based on additional assumptions considered in
each paper. In particular, we point out that~\cite{wu2019randomly} achieve
convergence in high-probability of the randomly initialized algorithm in
$O(\sqrt{n})$ steps along with proving minimaxity of the EM algorithm for the
well-defined symmetric 2-GM model.
Minimax optimality of generalized mixtures (multiple components ($K>2$) or
non-Gaussian mixtures) remains an open question.

Despite interest in producing unified results for
mixtures from exponential families or certain sub-families in the past, little
work has covered global convergence results for EM outside of Gaussian mixture
models. The most relevant paper to non-Gaussian models
is~\cite{redner1984mixture}. The authors present some initial results on
mixtures of exponential families, where the authors show that the
well-initialized EM algorithm converges to the true MLE if the Fisher
information is positive-definite. The main drawback to this result is the
un-verifiable condition of the Fisher information, and arguably harder to
verify, sufficiently close initialization of the algorithm.
This work aims to provide a new direction of analysis that can help close the
long-standing gap in the analysis of EM-based approaches.

\section{Main Results}\label{sec:results}
This section provides mathematical guarantees of the EM algorithm
applied to a balanced mixture of two exponentials to act as a generalized analogy of
the balanced symmetric Gaussian mixture case with minimal assumptions, thus filling in a
foundational gap in the existing literature at the cross section of mixture
models and iterative algorithmic estimators.

Under the set-up defined in Section~\ref{sec:fmms}, we let $(X,Z)$ be a pair of
random variables such that $X$ is generated from one of two exponential
distributions determined by the latent variable $Z$ which represents the
component of the mixture. The mixture distribution can then be written as
\begin{align}
  \label{eq:mix_dens}
  P_{\beta^*} = \frac{1}{2} \text{Exp}(\beta^*) + \frac{1}{2}\text{Exp}(\beta^*/\alpha ),
\end{align}
where $\beta^*$ is the scale parameter of the exponential density and
$\alpha >1$ is some fixed constant that defines the shift in parameter for the
second component of the mixture. There is no symmetry analogy that can be made
with Exponential distributions by definition, and so we propose a multiplicative
relationship between the parameters as a substitute.
It should be clear that assuming $\alpha > 1$ does not constrain our
analysis, for whenever $\alpha < 1$, we can simply redefine $\alpha' =
\frac{1}{\alpha}$ such that the scale parameter of the second component becomes
$\beta^*/\alpha'$. Here we consider only balanced mixtures, that is to say, $Z \sim
\text{Ber}(\frac{1}{2})$, to simplify our formulation in a manner that fits with
many of the earlier studies of Gaussian mixtures. However, as we will show in
Section~\ref{sec:sims}, our results appear to hold for a more general class of distributions.
Note that with these assumptions, we can rewrite the joint density as
\begin{align}
  \label{eq:mix_density}
  p_{\beta^*}(x,z) =
  \left(\frac{e^{-x/\beta^*}}{\beta^*}\right)^{z}
  \left( \frac{\alpha e^{-x\alpha/\beta^*}}{\beta^*} \right)^{(1-z)}.
\end{align}
For the remainder of the paper we let $\beta^*$ refer to the true value of the
parameter of interest. We let $\beta_t$ refer to the population EM
estimates, that is, the estimates given by EM if we had infinite sample size,
and $\hat{\beta}_t$ to the finite-sample estimates.
We capture the assumption that the data observed is generated by a mixture of
two exponentials below for ease of reference.
\begin{assumption}[DGP]
  \label{as:dgp}
  The observations $X_1, \ldots, X_n$ are i.i.d.\ samples from
  the following model:
\begin{align*}
  p_{\beta^*}(x,z) =
  \left(\frac{e^{-x/\beta^*}}{\beta^*}\right)^{z}
  \left( \frac{\alpha e^{-x\alpha/\beta^*}}{\beta^*} \right)^{(1-z)},
\end{align*}
where $\beta^*>0$, $\alpha >1$ and $Z \sim \text{Ber}(\frac{1}{2})$.
\end{assumption}
Now, we can identify a
closed-form solution for each step of the EM algorithm, captured in the
following proposition.
\begin{proposition}
  \label{prop:closed_form}
  Suppose Assumption~\ref{as:dgp} holds. Then, each step of the finite-sample EM
  Algorithm can be written in closed form as
\begin{align}
  \label{eq:beta_hat}
    \hat{\beta}_{t+1}
    = \frac{\alpha}{n} \sum_{i=1}^n X_i
  - \frac{(\alpha - 1)}{n}
  \sum_{i=1}^n X_i \big(1+\alpha e^{(1-\alpha)X_i / \hat{\beta}_t}\big)^{-1}.
\end{align}
In the limit as $n \rightarrow \infty$, we recover the closed-form for the
population EM to be
\begin{align}
  \label{eq:beta_pop}
  \beta_{t+1}
  = \alpha \mathbb{E}[X]
  - (\alpha - 1)\mathbb{E}\bigg[X \big(1+\alpha e^{(1-\alpha)X / \beta_t}\big)^{-1}\bigg].
\end{align}
\end{proposition}
The proof of this proposition, as well as all other results in this
section are deferred to Appendix~\ref{app:main_proofs}. Additional technical
results used in proving the results of this section that may be of independent
interest are provided in Appendix~\ref{app:technical}.
Proposition~\ref{prop:closed_form} shows that the expectation and maximization
steps can be reduced to a single closed-form update equation that allows us to
directly analyze the behavior of subsequent (finite-sample or population)
estimates, provided a current estimate. We will study the properties of
these update equations to set up a recursion for which we can establish
convergence to show that both the population and finite-sample EM estimates
converge to the true parameters next.

\subsection{Population EM Convergence Analysis}
\label{sec:pop-results}
Given the simplified recursive equations in Proposition~\ref{prop:closed_form},
it is possible to now understand how the sequence of EM estimators evolve,
providing a basis for determining the limiting value of the sequence as well as
the rate of convergence. We begin with convergence in the population EM
case (as defined in~\eqref{eq:beta_pop}), which will help guide our
understanding of the EM algorithm as well as provide a motivation
for how one may expect the empirical EM algorithm to evolve.

\begin{theorem}[Recursive bound on population EM]
  \label{thm:pop_recursion}
  Suppose Assumption~\ref{as:dgp} holds and some initial estimate $\beta_{0}>0$ is provided.
    For all $t \geq 0$, the recursion of the population EM~\eqref{eq:beta_pop} satisfies
    \begin{equation*}
        |\beta_{t+1} - \beta^*| \leq (1-\kappa_{\alpha})|\beta_t-\beta^*|,
    \end{equation*}
    where $\kappa_\alpha = \frac{\alpha+1}{2} - \frac{\alpha-1}{2}\left(\frac{3}{e} + \frac{1}{\alpha} \right)$.
\end{theorem}
Notice that as stated, Theorem~\ref{thm:pop_recursion} does not always guarantee
convergence. Rather, $\kappa_{\alpha} \in (0, 1)$ must hold for the result of
Theorem~\ref{thm:pop_recursion} to be non-trivial.
In terms of our model parameters, this translates to constraining $\alpha \in (1, \alpha_{\max})$
where $\alpha_{\max} \approx 11.49$. Under this additional constraint, we can
get the following convergence of the population EM iterates.
\begin{cor}[Convergence of the Population EM]
  \label{cor:pop_conv}
  Suppose all the conditions of Theorem~\ref{thm:pop_recursion} hold.
  Additionally, let $\alpha \in (1, \alpha_{\max})$. Then,
  $|\beta_t - \beta^*| \leq \varepsilon$
  if
  \begin{align*}
    t \geq
    \frac{\log\left(\frac{|\beta_0-\beta^*|}{\varepsilon}\right)}
    {\log ((1-\kappa_{\alpha})^{-1})},
  \end{align*}
  where $\kappa_\alpha = \frac{\alpha+1}{2} - \frac{\alpha-1}{2}\left(\frac{3}{e} + \frac{1}{\alpha} \right)$.
\end{cor}
At this point, it is worth discussing the implications of the range of $\alpha$
that allows for convergence of the population EM iterate.
Let us start by considering the densities of each of the two components in
\eqref{eq:mix_density}. As $\alpha$ increases in value, the second
component, with scale parameter $\beta^*/\alpha$, concentrates towards zero
exponentially fast. The intersection point of the two components, thus, becomes
a point of almost perfect separation between the two densities since the second
component places almost no weight to the right of the intersection point, while
the first component, provided $\beta^*$ is not too large itself, places a
significant probability weight to the right. This is illustrated in
Figure~\ref{fig:alpha} where $\beta^*$ is fixed and different $\alpha$ values
are considered.
It is visible from Figure~\ref{fig:alpha} that the density of the
mixture is increasingly overwhelmed by the second component, essentially driving
the weight of the first component to zero despite observations coming from both
components, as shown in Figure~\ref{fig:alpha-sample}.
Figure~\ref{fig:alpha-sample} shows that the samples are essentially completely
disjoint in which component they are generated from.
As such, the algorithm is unable to interpret $\beta^*$ clearly from the data,
making convergence of the iterates impossible as $\alpha \rightarrow \infty$.
We leave the heuristic discussion here and refer the reader to the
discussion following the proof of Theorem~\ref{thm:pop_recursion} for the
mathematical formulation of this phenomenon.
\begin{figure}[h!]
    \centering
    \includegraphics[width=0.8\textwidth]{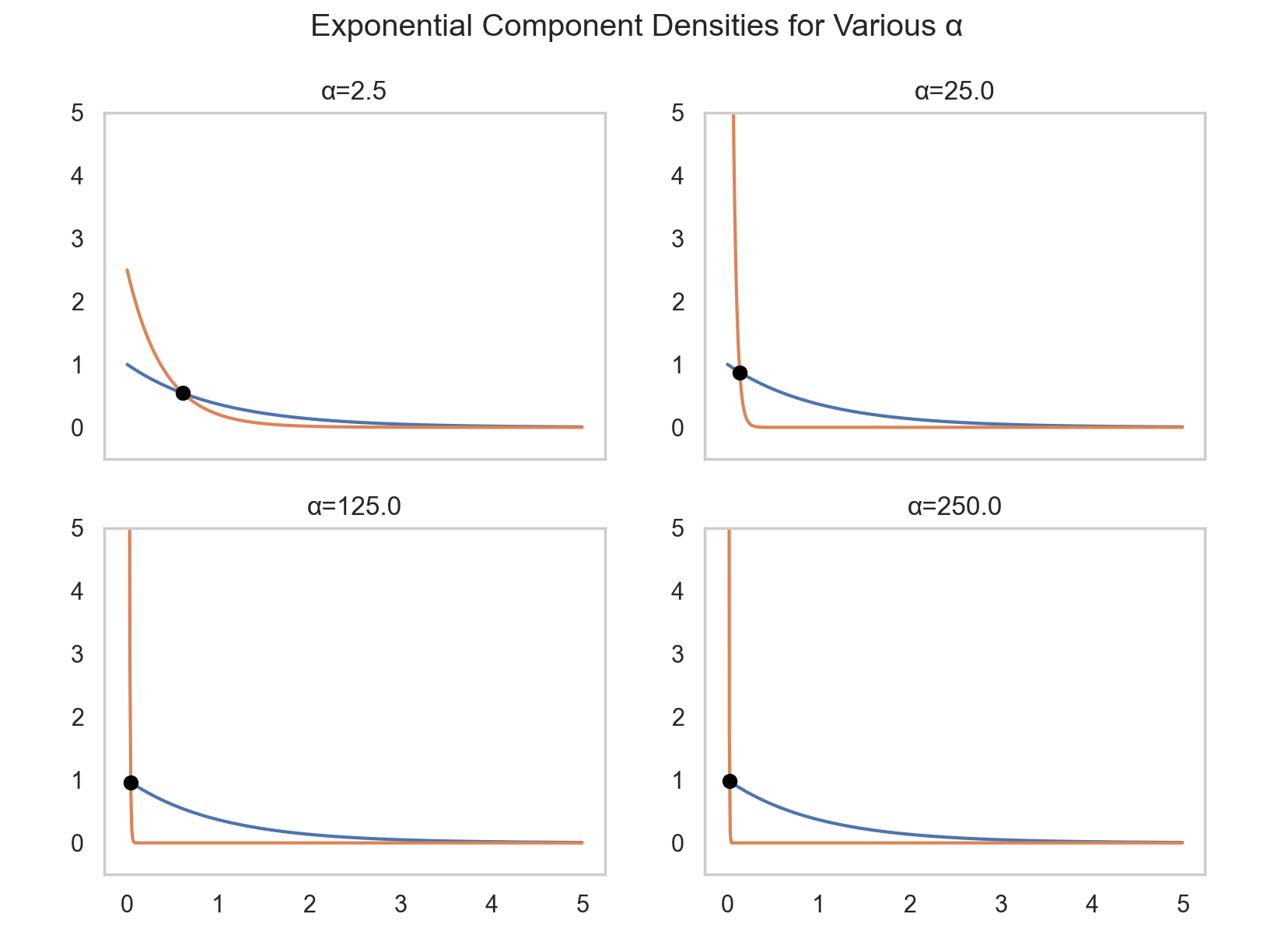}
    \caption{Exponential densities with scale parameter $\beta^*=2$ in blue and
      $\frac{\beta^*}{\alpha}$ in orange.
       The black point shows the intersection of these two densities, which tends to $0$ as $\alpha \rightarrow \infty$.}
    \label{fig:alpha}
\end{figure}

\begin{figure}[h!]
    \centering
    \includegraphics[width=0.6\textwidth]{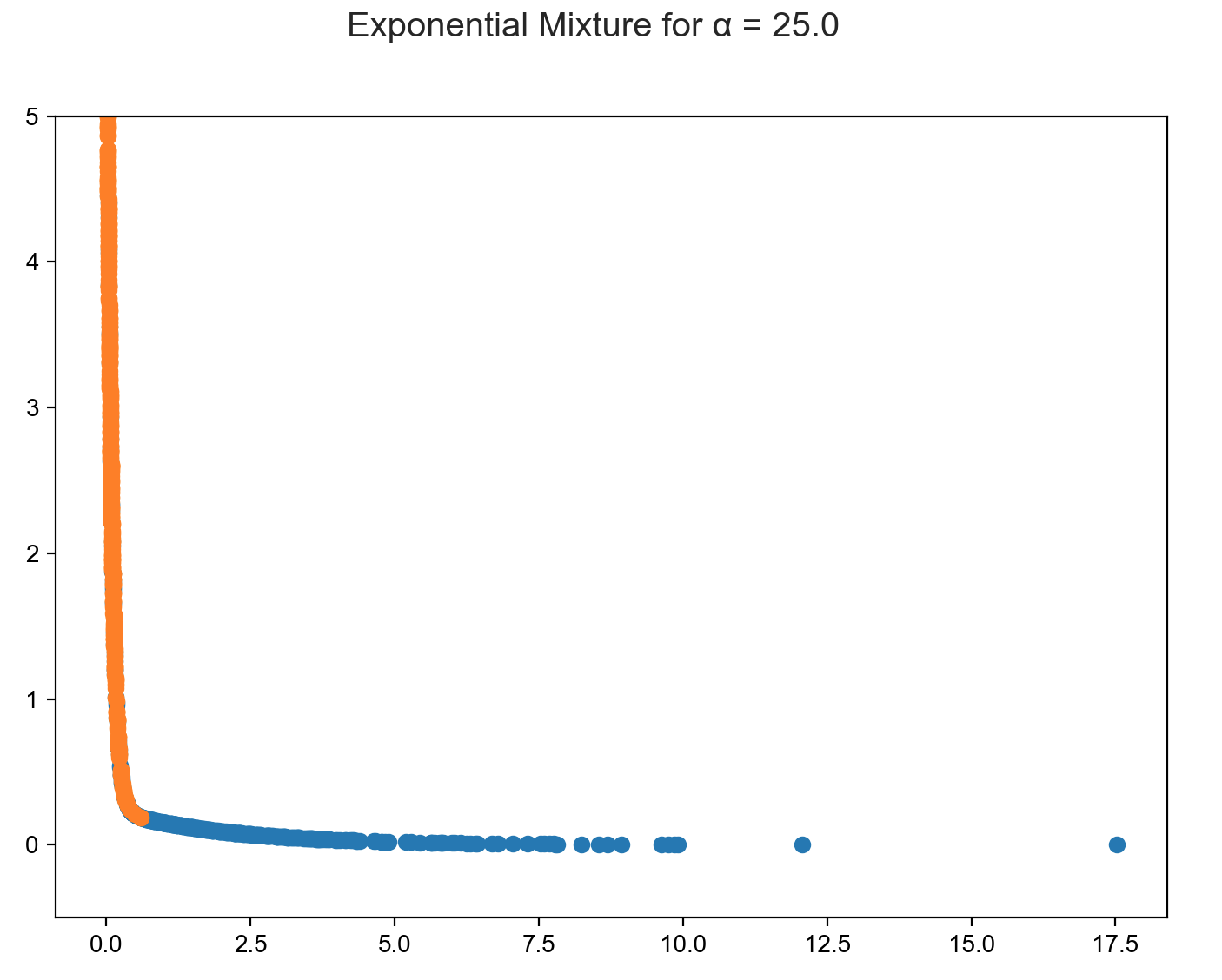}
    \caption{Scatter plot of sample size $n=1000$ from the balanced exponential
      mixture with $\beta^*=2$ and $\alpha=25$. Points in orange come from the component
      with scale $\frac{\beta^*}{\alpha}$, points in blue come from the component with
      scale $\beta^*$.}
    \label{fig:alpha-sample}
\end{figure}

It may also be of interest to note that having a restriction on $\alpha$ is not
completely unexpected as similar constraints exist for the symmetric mixture of
two Gaussians in the form of minimum separation. All
existing theoretical results on Gaussian mixtures place a lower bound on the
distance between the means of the two components in
order to show any valid convergence
\cite[amongst others]{wu2019randomly, kwon2020algorithm, daskalakis2017ten,
xu2016global}. In fact, \cite{wu2019randomly}
show that in the case that the means of both components
approach zero (that is, the data is actually only generated from a standard
normal distribution), the EM has a sub-optimal rate of convergence. We see in
Figure~\ref{fig:alpha-sample}
that $\alpha$ plays a similar role in that as it tends to infinity, the data
from a single component overwhelms the estimator and so the algorithm fails
to identify the two mixture components. The key distinction here with respect to
the Gaussian mixture case is that $\beta^*$ and $\alpha$ simultaneously affect
both the location and shape of the components. We suspect this confounding
effect may play a role in the admissibility range of $\alpha$, that under the
current approach, is not distinguishable.
As a result, it is entirely possible that the identified range for $\alpha$ is not
be optimal, in that, the upper limit could be an artifact of the proof
methodology or the parametrization of the mixture model (i.e., a user may be
interested in placing constraints on $\beta^*$ as opposed to $\alpha$, which
could change the range of admissibility as well as the interpretation of the
result). The question of an optimal interval for $\alpha$ (or alternatively
$\beta^*$) remains open and is an area of interest for future work.

\subsection{Finite-Sample Convergence Analysis}
\label{sec:finite-results}
Now, we turn to the finite-sample EM iterates, which through their high-probability convergence
to the population iterates, we will show approximate the true parameters.
\begin{theorem}[Convergence of finite-sample EM iterates]
  \label{thm:empirical_bound}
  Let $\hat{\beta}_0 = \beta_0 > 0$, that is, let the initial estimate for the
  population and finite-sample EM iterates be equal.
  Further, assume that $n \gtrsim \text{poly}(\alpha; \varepsilon)$.
  Then for all $t\geq0$, with probability $1-\gamma$,
  \begin{align*}
    |\hat{\beta}_{t} -\beta_{t}|
    \lesssim
    \beta^* \Big(\sqrt{\frac{\log(2/ \gamma)}{n}}
    + \frac{\log(2/ \gamma)}{n}\Big).
  \end{align*}
\end{theorem}
Theorem \ref{thm:empirical_bound} shows that with high probability, the
distance between the finite-sample and population iterates can be controlled. In
particular, we note that the rate at which the distance between the two iterates
goes to zero reflects the standard concentration rates for sub-Exponential
random variables. That is, we see the $n^{-1/2}$ rate for small deviations that
fall into the sub-Gaussian regime and the $n^{-1}$ rate for the larger
deviations that fall into the sub-Exponential regime. Thus, a mixture model of
sub-Exponential distributions maintains a sub-Exponential rate of convergence
and so we conjecture that this rate would be minimax optimal up to some
constants that may depend on the parameters of the mixture model.

It is also relevant to note that Theorem~\ref{thm:empirical_bound} includes a
constraint on $n$ in terms of $\alpha$ and $\varepsilon$ (through $t$, as
established in Corollary~\ref{cor:pop_conv}) that is distinct from the
constraint on $\alpha$ identified in Corollary~\ref{cor:pop_conv}. This
constraint is purely an artifact of the recursion formulation for convergence
and does not impose a restriction on $\alpha$ in the way
Theorem~\ref{thm:pop_recursion} does. In essence, one should treat this as a minimum
number of samples needed to achieve convergence of the finite-sample EM iterates
for a given $\alpha$ and error tolerance $\varepsilon$ rather than a limit on
the true value of $\alpha$. We refer the reader to the proof of the theorem in
Appendix~\ref{app:main_proofs} for the exact lower bound on $n$ in terms of
$\alpha$ and $t$. From the simulation study in Section~\ref{sec:sims}, it
appears that the lower bound on $n$ is not necessarily optimal and there may be
room for improvement. Further discussion regarding this is deferred to
Section~\ref{sec:sims}.

We can now bring together the results of Theorem~\ref{thm:pop_recursion}
and~\ref{thm:empirical_bound} to show convergence of the sample EM
iterates to the true parameter.
\begin{cor}\label{cor:empirical_rate}
  Suppose all assumptions of Theorem~\ref{thm:empirical_bound} hold.
  Furthermore, let
  $t \geq \log(\frac{|\beta_0-\beta^*|}{\varepsilon})/\log (\frac{1}{1-\kappa_{\alpha}})$.
  Then, with probability $1-\gamma$
  \begin{align*}
    |\hat{\beta}_{t} - \beta^*|
    \lesssim
    \beta^*
    \Big(\sqrt{\frac{\log(2/ \gamma)}{n}}
    + \frac{\log(2/ \gamma)}{n}\Big)
    +\varepsilon.
  \end{align*}
\end{cor}
The result of Corollary~\ref{cor:empirical_rate} follows directly from
Theorems~\ref{thm:pop_recursion} and~\ref{thm:empirical_bound}. The rate of
convergence, as highlighted earlier, is inline with the rates for
sub-Exponential random variables.
We note that the number of iterations required for convergence in
Corollary~\ref{cor:empirical_rate} scales at least as $\log n$, in order for the
approximation error to be of higher-order than the estimation error, which
is comparable to the number of iterations required (in the best case) for Gaussian
mixtures as well \cite{wu2019randomly, xu2016global}. We additionally highlight that
the error rate has an inverse dependency on the magnitude of the true parameter
$\beta^*$ than typically found in Gaussian mixture (for the mean parameter)
analysis, however, we believe this is largely attributed to the differing roles
that the parameter of interest plays in the Gaussian density versus Exponential
densities, as discussed in Section~\ref{sec:pop-results} in connection with
Corollary~\ref{cor:pop_conv}.
Corollary~\ref{cor:empirical_rate} shows that the EM algorithm can adapt to
non-Gaussian mixtures well without sacrificing the rate of convergence or the
number of iterations required to achieve a specific error tolerance. This opens
many avenues for generalizing the applicability of EM to various classification
and regression problems.

\section{Simulations}\label{sec:sims}

In order to corroborate our theoretical results, we simulate 2-component
mixtures of exponentials following the model~\eqref{eq:mix_density} and
algorithm (Alg~\ref{alg:EM}) detailed in Section \ref{sec:results} for a range
of values of $\alpha$ and $\beta^*$.
For each experiment, we randomly initialize $\hat{\beta}_0 \in (0,10)$ and then
run the algorithm using the closed-form expression in
Proposition~\ref{prop:closed_form} for the number of steps necessary to achieve
an error of $\varepsilon \leq 10^{-2}$, calculated according to
Corollary~\ref{cor:empirical_rate}.
All replication files (written in Python) for the simulations presented here can
be found on Github at
\url{https://github.com/kdullerud/EM-convergence-exponential}.

\begin{figure}[h!]
  \centering
  \includegraphics[width=0.8\textwidth]{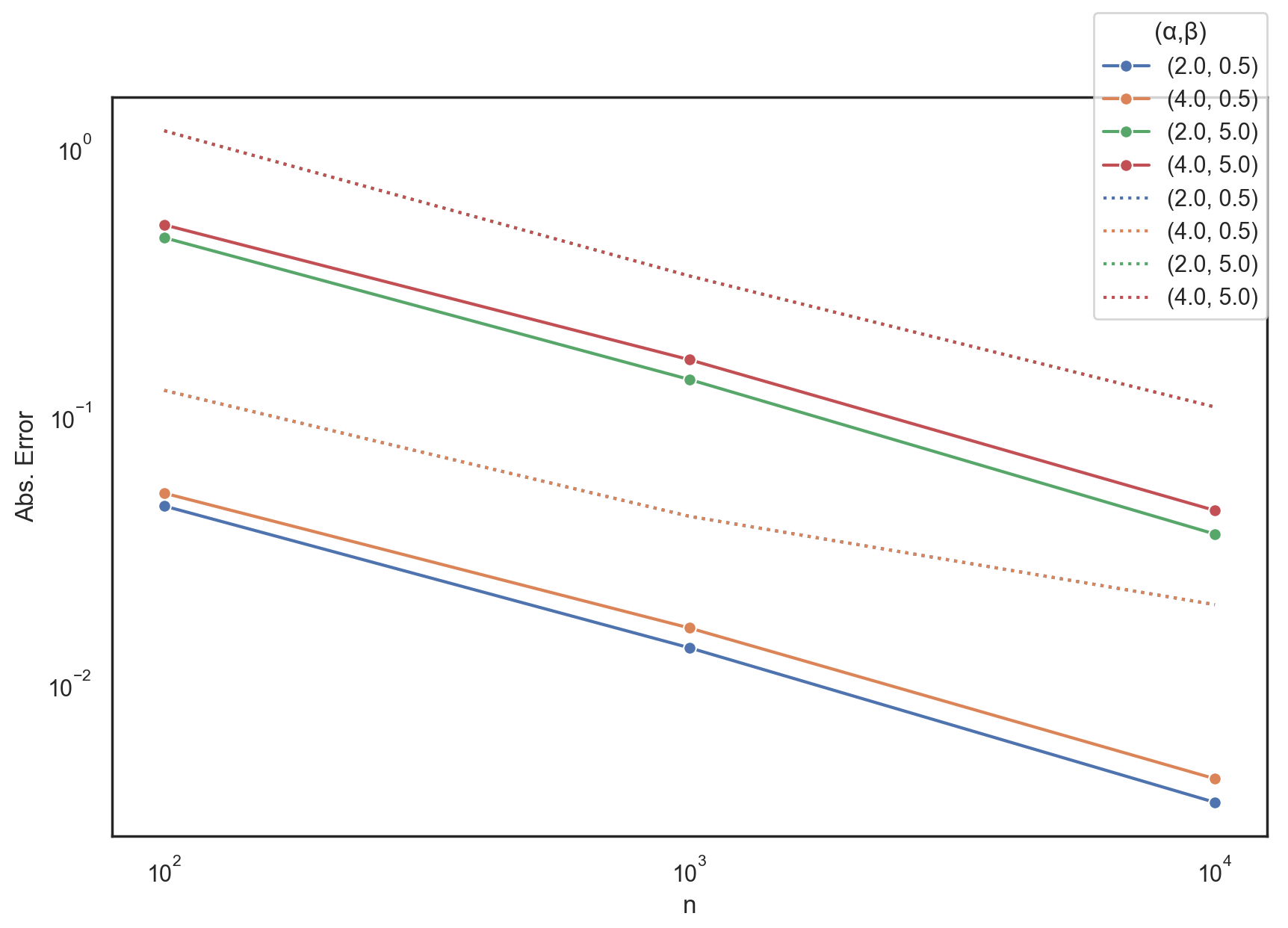}
  \caption{Plot of the absolute statistical error $|\hat{\beta}_T-\beta^*|$
    versus sample size $n$, where $\hat{\beta}_T$ is the converged iterate.
    The solid lines show the empirical absolute statistical error, averaged over
    $50$ independent runs. The dotted lines show the corresponding theoretical
    bound found in Corollary \ref{cor:empirical_rate}.}
  \label{fig-sample}
\end{figure}
Figure \ref{fig-sample} shows the absolute error of the converged
iterate of the algorithm versus the sample size $n$. We see that the simulated
results, shown in solid lines fall below the bound found in Theorem
\ref{thm:empirical_bound}, shown in dotted lines, but follow the same general
trend as a function of $n$. We first note that our bound seems to be relatively
tight as a function of the parameters $\alpha$ and $\beta^*$ and that it
correctly identifies the influence of $\beta^*$ on the absolute error.
We highlight here that $\alpha$ seems to play little role in the rate,
which corroborates the findings of Theorem \ref{thm:empirical_bound}.

It is important here to note the impact of the assumption in
Theorem~\ref{thm:empirical_bound} and Corollary~\ref{cor:empirical_rate} that requires
$\sqrt{n} \gtrsim poly(\alpha; \varepsilon)$. We can see directly from
Figure~\ref{fig-sample} that the bound is not optimal. As an example, consider
the case where $\alpha = 4$, and an error threshold of $\varepsilon=0.01$.
According to Theorem~\ref{thm:empirical_bound}, we would require sample size
$n \gtrsim 10^{14}$ which is both computationally expensive to work with and can
be unattainable in many datasets. Moreover, Figure~\ref{fig-sample} clearly
shows that the algorithm is able to converge with significantly smaller sample
sizes. Thus, despite not taking a large enough $n$ according to the theory,
the bound of Theorem~\ref{thm:empirical_bound} shows relatively good control on
the error in Figure~\ref{fig-sample}, giving us reason to believe that there is
room to improve in the minimum sample size requirement for the conclusions of
Theorem~\ref{thm:empirical_bound} and Corollary~\ref{cor:empirical_rate} to be
vaild.
\begin{figure*}[t!]
\centering
\begin{subfigure}[t]{0.48\textwidth}
    \centering
    \includegraphics[height=1.65in]{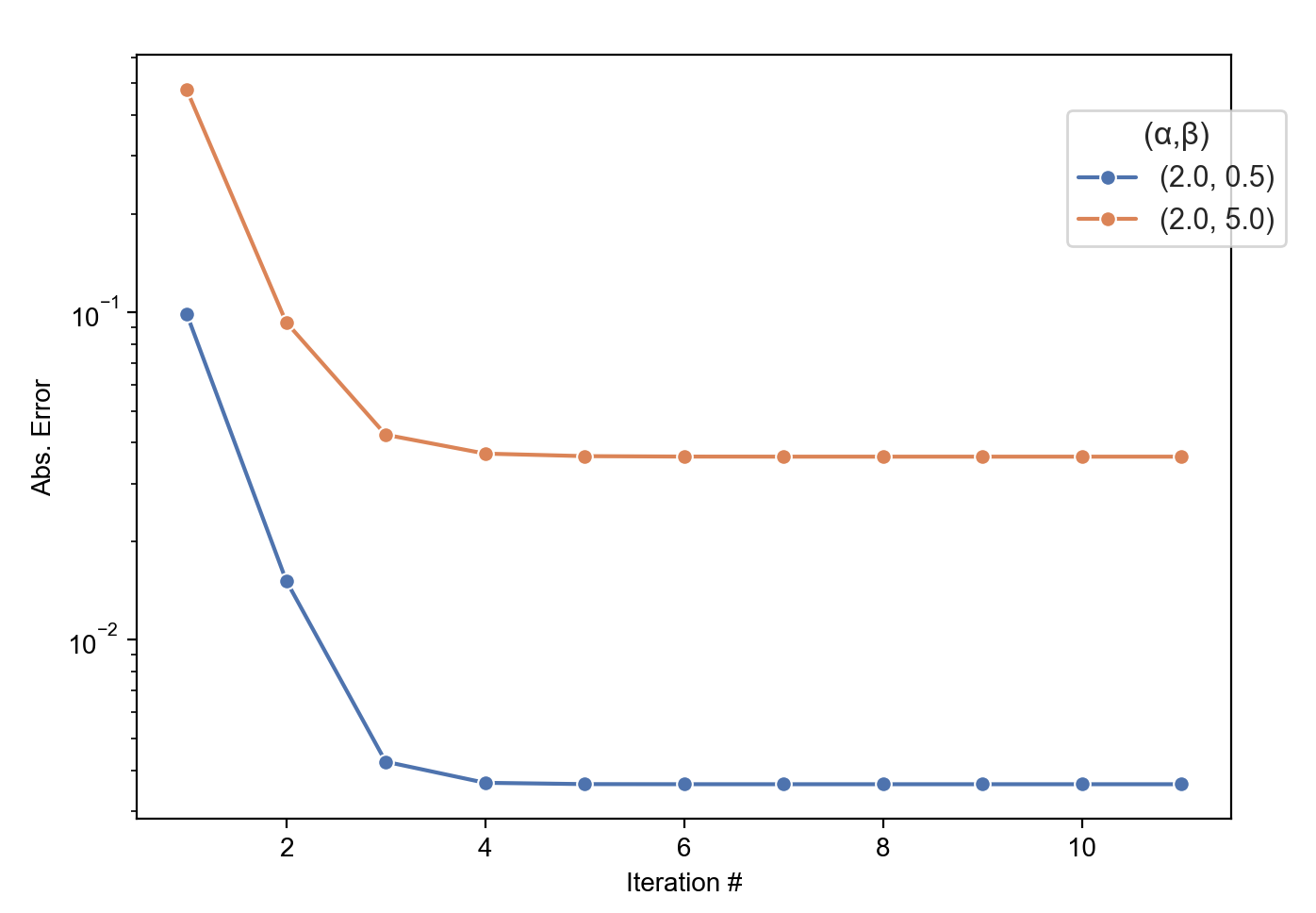}
\end{subfigure}
\begin{subfigure}[t]{0.48\textwidth}
    \centering
    \includegraphics[height=1.65in]{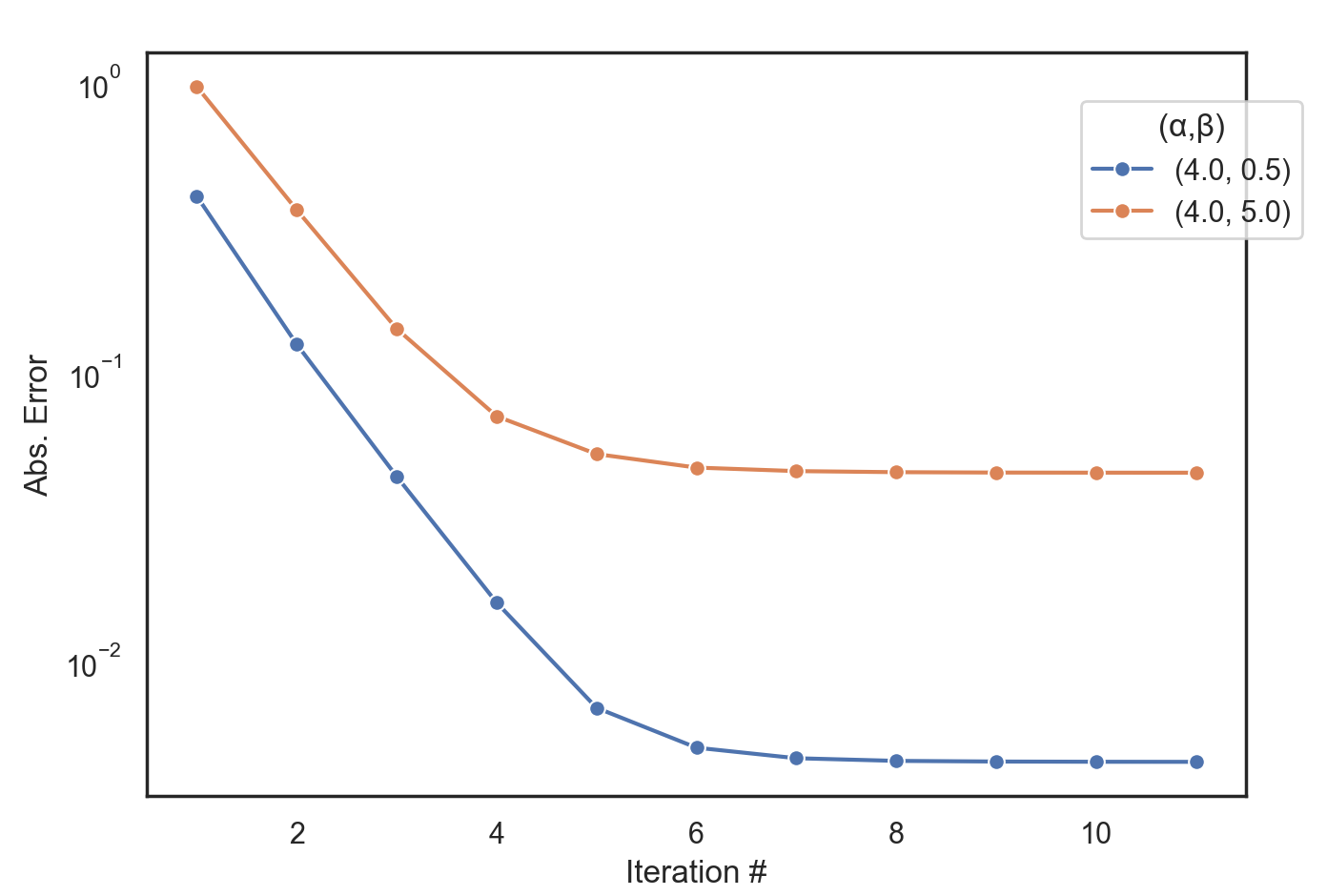}
\end{subfigure}
    \caption{Plot of iteration number versus statistical error
      $|\hat{\beta}_t-\beta^*|$ at each iteration $t$ for sample size $n=10^4$
      and total number of iterations $T=11$. Simulations were averaged over $50$
      independent runs. Plot for simulations with $\alpha = 2$ on left, $\alpha =4$ on right.}
    \label{fig-iters}
\end{figure*}

We further investigate our results on the number of iterations for convergence,
shown in Figure~\ref{fig-iters}. Our results guarantee that for $t \geq
\log(\frac{|\beta_0-\beta^*|}{\varepsilon})/\log (\frac{1}{1-\kappa_{\alpha}})$,
or after plugging in the corresponding values for $\alpha =2$, $\alpha = 4$ in
our simulations, we achieve the error bound found in
Corollary~\ref{cor:empirical_rate}
for $t \geq 6$, $t \geq 11$ respectively. In inspecting Figure~\ref{fig-iters},
we see that there is indeed not much room for improvement here. Though all three
parameter schemes reach their proven error before the required number of
iterations, our estimation of the total number of iterations needed is based on
the worst-case value of $|\beta_0-\beta^*|$. Thus, our result on the minimum
number of iterations needed to achieve convergence within the error tolerance,
as shown in Corollary~\ref{cor:empirical_rate}, should always upper bound the
empirical minimum number of iterations needed.

In addition to corroborating the theoretical results shown in this paper, we
simulate an unbalanced 2-component mixture of exponentials in order to
understand the generalizing potential of the current theory and to motivate
future work.
We run each experiment in the same fashion as the balanced mixture model, with
mixing proportions replaced by $0.3$ and $0.7$ for components with mean
$\beta^*$ and $\beta^*/\alpha$, respectively. We assume these mixing proportions
to be known.
Figure \ref{fig-unbalanced} shows the absolute statistical error which behaves
very similarly to the error in Figure \ref{fig-sample}.
Given the similarity in results, perhaps with even faster error decay in the
unbalanced case, empirically, there is reason to believe the results of
Section~\ref{sec:results} are applicable to more general settings.
\begin{figure}[h!]
  \centering
  \includegraphics[width=0.75\textwidth]{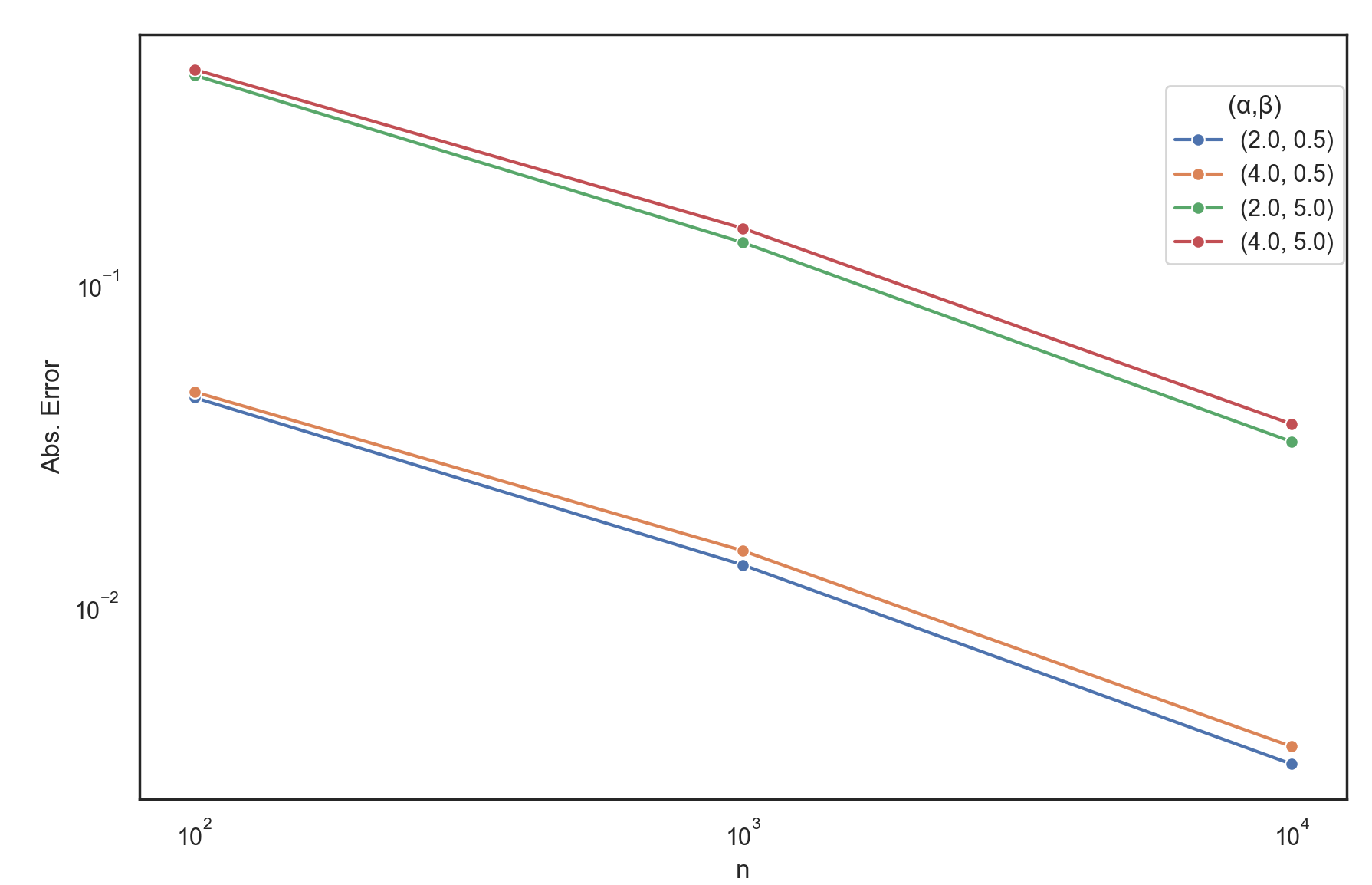}
  \caption{Plot of the absolute statistical error $|\hat{\beta}_T-\beta^*|$
    versus sample size $n$ in the unbalanced mixture case (mixing proportions
    $(0.3, 0.7)$), where $T$ is the converged iterate.
    The solid lines show the empirical absolute statistical error, averaged
    over $50$ independent runs.}
  \label{fig-unbalanced}
\end{figure}
Investigating these
generalizations can be crucial in building towards a unified theory of
convergence of the EM algorithm for mixture models of exponential families.

\section{Conclusion \& Future Work}\label{sec:conclusion}
The results presented in this paper provide a foundation for the investigation
of generalized mixture models from a theoretical perspective that could
broaden the application of EM. We have shown universal asymptotic consistency of
EM estimates for a balanced 2-component mixture of Exponentials under random
initialization. Furthermore, we provide rates of convergence for the empirical
algorithm that require the number of iterations to grow at most logarithmically
in the number of data points. There remain a number of interesting avenues for
future work, as few of which are outlined below.
\begin{itemize}
\item \textbf{High-dimensional mixture models}: It has recently been a point of
  interest in statistics to study high-dimensional models. We expect that
  the results of this paper will generalize directly for high-dimensional
  Exponential mixtures, provided there is independence across dimensions.
  However, ultra high-dimensional settings (like when $d \propto n$) may require
  a completely different perspective of analysis~\cite{candes2020phase}.

\item \textbf{$K>2$ components}: Beyond exploring the unbalanced two-component
  case, as mentioned earlier, a further natural extension would be to
  consider more that two components. Some existing work along these lines for
  Gaussian mixtures~\cite{jin2016local} can be informative on what may be achievable.

\item \textbf{Federated learning}: Recent statistical advances have extended the
  study of mixture models into the federated learning framework, which allows
  for privacy considerations while leveraging statistical power of averaging
  larger populations~\cite{tao2024convergence, reisizadeh2024mixture,
    wu2023personalized}. Extensions of mixtures of Exponentials to federated
  learning can have implications for financial applications where often the
  tail decay is non-Gaussian in nature.
\item \textbf{Limiting assumptions}: This paper also highlights the importance
  of the separation of mixture components in proving global convergence results,
  which has been discussed in the literature previously as
  well~\cite{daskalakis2017ten, wu2019randomly}, albeit from the perspective of
  minimum separation. The present work provides a different perspective on such
  separation conditions that apply for non-symmetric distributions. Further
  investigation into optimality of such assumptions and minimax bounds on more
  general classes of mixture models would be paramount to identifying
  applicability of iterative algorithms, including EM, to a general class of
  functions.
\item \textbf{Unified theory for exponential families}:
  The question of developing a more unified framework on the global convergence
  conditions of the EM algorithm for mixtures of exponential family densities
  remains of great interest and has been largely unexplored in the past forty
  years~\cite{wu1983convergence}.
\end{itemize}

\begin{credits}
\subsubsection{\ackname}
The authors would like to thank Paul Freulon for insightful comments on an
earlier draft of this paper.

\subsubsection{\discintname}
The authors have no competing interests to declare that are
relevant to the content of this article.
\end{credits}

\bibliographystyle{splncs04}
\bibliography{ref}

\begin{subappendices}
  \renewcommand{\thesection}{\Alph{section}}%
  \section{Proofs of main results}
  \label{app:main_proofs}
  Here we provide detailed proofs of each of the results in Section~\ref{sec:results}.
  \subsection*{Proof of Proposition~\ref{prop:closed_form}}
We begin from the complete log-likelihood of the pairs $(X_i,Z_i)_{i=1}^n$,
ignoring any constants that do not affect the maximization step:
\begin{align*}
\ell_{\beta}(X,Z)
  = -n\log \beta +
  \sum_{i=1}^n \left[ \frac{Z_i(\alpha-1) X_i}{\beta} - \frac{\alpha X_i}{\beta}\right] .
\end{align*}
Then taking the expectations with respect to the estimated distribution for each
$Z_i$ conditional on $\hat{\beta}_t$, the current parameter estimate,
\begin{align*}
  Q(\beta|\hat{\beta}_t)
  &= -n\log\beta +
    \sum_{i=1}^n \sum_{z_i=0}^1 \mathbb{P}(Z_i=z_i|X_i, \hat{\beta}_t)\left(
\frac{z_i(\alpha-1) X_i}{\beta} - \frac{\alpha X_i}{\beta}\right) \\
  &= -n\log\beta - \frac{\alpha}{\beta}\sum_{i=1}^n X_i + \frac{1}{\beta}\sum_{i=1}^n \mathbb{P}(Z_i = 1|X_i, \hat{\beta}_t)(\alpha-1) X_i.
\end{align*}
Now, we can approximate $\mathbb{P}(Z_i = 1|X_i, \hat{\beta}_t)$ in terms of the
data and $\hat{\beta}_t$ by using Bayes' formula:
\begin{align*}
    \mathbb{P}(Z_i=1|X_i, \hat{\beta}_t)
  &= \frac{\mathbb{P}(X_i|Z_i, \hat{\beta}_t)\mathbb{P}(Z_i|\hat{\beta}_t)}
  {\mathbb{P}(X_i|\hat{\beta}_t)}
  \\
  &
  = \frac{\frac{e^{-X_i/\hat{\beta}_t}}{2\hat{\beta}_t}}
  {\frac{e^{-X_i/\hat{\beta}_t}}{2\hat{\beta}_t} + \frac{\alpha e^{-X_i \alpha / \hat{\beta}_t}}
  {2\hat{\beta}_t}}
  = (1+ \alpha e^{(1-\alpha) X_i /\hat{\beta}_t})^{-1}
\end{align*}
Thus, plugging this back into the $Q-$function, we have that
\begin{align*}
Q(\beta|\hat{\beta}_t)
  = -n\log\beta
  - \frac{\alpha}{\beta} \sum_{i=1}^n X_i
  + \frac{(\alpha - 1)}{\beta}
  \sum_{i=1}^n X_i(1+ \alpha e^{(1-\alpha) X_i /\hat{\beta}_t})^{-1}.
\end{align*}
Now, we are left only to solve the maximization step, that is, maximize
$Q(\cdot|\cdot)$ with respect to $\beta$.
\begin{align*}
    \frac{dQ}{d\beta}
    = -\frac{n}{\beta}
    + \frac{\alpha}{\beta^2} \sum_{i=1}^n X_i
    -\frac{(\alpha - 1)}{\beta^2}
    \sum_{i=1}^n X_i(1+ \alpha e^{(1-\alpha)X_i / \hat{\beta}_t})^{-1}
\end{align*}
and thus solving at $\frac{dQ}{d\beta} = 0$ yields
\begin{align*}
    \hat{\beta}_{t+1}
  = \frac{\alpha}{n} \sum_{i=1}^n X_i
  - \frac{(\alpha - 1)}{n}
  \sum_{i=1}^n X_i (1+ \alpha e^{(1-\alpha)X_i / \hat{\beta}_t})^{-1} .
\end{align*}
From here, it is straightforward to take the limit as $n \rightarrow
\infty$ and obtain the desired result for the population EM case. \qed

\subsection*{Proof of Theorem~\ref{thm:pop_recursion}}
We note that this proof is similar to that of \cite[Theorem 1]{daskalakis2017ten}
in that we also proceed by using the Mean Value Theorem.
The application, however, of such an approach to exponential distributions and
the control of the derivative is novel.
For simplicity, let $f(\cdot, \cdot)$ denote the function defined by each step under the population EM, such that
\begin{align*}
\beta_{t+1}
  = f(\beta_t, \beta^*)
  = \alpha \mathbb{E}_{\beta^*}[X]
  - (\alpha - 1)\mathbb{E}_{\beta^*}[X \cdot (1+\alpha e^{(1-\alpha)X / \beta_t})^{-1}],
\end{align*}
where by an abuse of notation, we let $\mathbb{E}_{\beta^*}[\cdot]$ refer to the
expectation under the mixture as defined in~\eqref{eq:mix_density}. Some
relevant properties of $f$ are captured in Lemma~\ref{lem:f_properties}.

Let us consider the case where $\beta_t < \beta^*$.
By the fact that $f$ is continuous with respect to its second argument and the
Mean Value Theorem, we know that there must exist a
$\tilde{\beta} \in [\beta_t, \beta^*]$
such that
\begin{equation*}
  \frac{f(\beta_t, \beta^*) - f(\beta_t, \beta_t)}{\beta^*-\beta_t}
  = \frac{\partial f(\beta_t,\mu)}{\partial \mu} \Bigg|_{\mu = \tilde{\beta}}
\end{equation*}
Recall that by definition $f(\beta_t, \beta^*) = \beta_{t+1}$
and by Lemma~\ref{lem:f_properties}, $f(\beta_t, \beta_t) = \beta_t$.
Thus,
\begin{align*}
  & \beta_{t+1} - \beta_t
    = (\beta^*-\beta_t) \left( \frac{\partial f(\beta_t,\mu)}{\partial \mu} \Bigg|_{\mu = \tilde{\beta}}\right)
    \geq (\beta^* - \beta_t) \left( \min_{\mu \in [\beta_t, \beta^*]} \frac{\partial f(\beta_t,\mu)}{\partial \mu} \right)
  \\
  & \iff \beta_{t+1}-\beta^* \geq (\beta_t - \beta^*) \left( 1- \min_{\mu \in [\beta_t, \beta^*]} \frac{\partial f(\beta_t,\mu)}{\partial \mu} \right).
\end{align*}
Invoking Lemma \ref{lem:f_properties} once again,
we have that since $f(\beta, \mu)$ is increasing with respect to $\beta$ for any
given $\mu$,
\begin{align*}
  \beta_{t+1} = f(\beta_t, \beta^*) \leq f(\beta^*, \beta^*) = \beta^*,
\end{align*}
by the assumption that $\beta_t < \beta^*$.
Thus it is true that $\beta_{t+1} \leq \beta^*$ and therefore:
\begin{align*}
    |\beta_{t+1}-\beta^*|
  = \beta^*-\beta_{t+1}
  \leq  |\beta_t-\beta^*| \left( 1- \min_{\mu \in [\beta_t, \beta^*]} \frac{\partial f(\beta_t,\mu)}{\partial \mu} \right).
\end{align*}
It remains to identify the range of the partial derivative for which a strict
contraction is obtained, which, by Lemma~\ref{lem:partial_deriv}, is lower bounded as
\begin{align*}
  \min_{\mu \in [\beta_t, \beta^*]} \frac{\partial f(\beta_t,\mu)}{\partial \mu}
  \geq
    \frac{\alpha+1}{2} - \frac{\alpha-1}{2}\left(\frac{3}{e} + \frac{1}{\alpha} \right).
\end{align*}
To see the case where $\beta_t > \beta^*$, we simply flip the order of the terms
in the use of the MVT (in the given interval), and follow the same steps as above
to arrive at the final expression.
\qed
\vspace{0.2cm}
\noindent
\textbf{Note:} In connection with the discussion regarding the range of
convergence with respect to $\alpha$ from Section~\ref{sec:pop-results},
observe that the second term of $f$ is approximately $(\alpha -
1)\E_{\beta^*}[X]$ for large alpha.
This implies that
\begin{align*}
  f(\beta_t, \beta^*)
  \approx
  \frac{1}{2}(\beta^* + \frac{\beta^*}{\alpha})
\end{align*}
for $\alpha$ sufficiently large. This eliminates the effect of the weights on the
iteration of the estimates for $\beta$ and leads to a formulation that does not
guarantee convergence.

\subsection*{Proof of Corollary~\ref{cor:pop_conv}}
By recursion of the statement of Theorem~\ref{thm:pop_recursion},
\begin{align*}
  |\beta_t-\beta^*|
  \leq (1-\kappa_\alpha)^t |\beta_0-\beta^*|.
\end{align*}
To ensure that this bound is at most $\varepsilon$,
we must satisfy
\begin{align*}
  (1-\kappa_\alpha)^T |\beta_0-\beta^*| \leq \varepsilon
\end{align*}
for some $T>0$.
Thus, for
\begin{align*}
  T \geq \log \left(
  \frac{|\beta_0-\beta^*|}{\varepsilon}\right) / \log \left(
  \frac{1}{1-\kappa_\alpha}\right)
\end{align*}
the conclusion follows.
\qed

\subsection*{Proof of Theorem~\ref{thm:empirical_bound}}
  We complete the proof in four steps.
  \subsubsection{1. Iterative structure:}
    We start by defining
          \begin{align*}
            g(X,\beta)
            = \alpha X - (\alpha-1)X(1+\alpha e^{(1-\alpha)X/\beta})^{-1},
          \end{align*}
          where $X \sim P_{\beta^*}$ as defined in~\eqref{eq:mix_dens}.
          Rewriting our iterates in terms of $g$, we have
          \begin{align*}
            \hat{\beta}_{t+1}
            &= \frac{1}{n}\sum_{i=1}^n g(X_i,\hat{\beta}_t)\qquad \text{and},
            \\
            \beta_{t+1}
            &= \mathbb{E}[g(X,\beta_t)].
          \end{align*}
          We will use this formulation to show concentration of the sample
          iterates around the population iterates in absolute value,
          \begin{align*}
            |\hat{\beta}_{t+1} - \beta_{t+1}|
            =
            \left|\frac{1}{n}\sum_{i=1}^n g(X_i,\hat{\beta}_t) - \mathbb{E}[g(X,\beta_t)]\right|,
          \end{align*}
          for any $t$.
          We start by adding and subtracting
          $\frac{1}{n}\sum_{i=1}^ng(X_i, \beta_t)$ inside the absolute value to
          get
        \begin{align}
          \label{eq:g_bound}
          &\left|\frac{1}{n}\sum_{i=1}^n g(X_i,\hat{\beta}_t) - \mathbb{E}[g(X,\beta_t)]\right|
            \nonumber
          \\
          &
           = \left|\frac{1}{n}\sum_{i=1}^n g(X_i,\hat{\beta}_t)
            -\frac{1}{n}\sum_{i=1}^ng(X_i, \beta_t)
            + \frac{1}{n}\sum_{i=1}^ng(X_i, \beta_t)
            - \mathbb{E}[g(X,\beta_t)]\right|
            \nonumber
          \\ &
               \leq   \left|\frac{1}{n}\sum_{i=1}^n g(X_i,\hat{\beta}_t)
               -\frac{1}{n}\sum_{i=1}^ng(X_i, \beta_t)\right|
               + \left| \frac{1}{n}\sum_{i=1}^ng(X_i, \beta_t)- \mathbb{E}[g(X,\beta_t)]\right|
               \nonumber
          \\ &
               \leq \underbrace{\frac{1}{n}\sum_{i=1}^n\left|g(X_i,\hat{\beta}_t)
               -  g(X_i,\beta_t)\right|}_{(I)}
               + \underbrace{\left| \frac{1}{n}\sum_{i=1}^n g(X_i, \beta_t)- \mathbb{E}[g(X,\beta_t)]\right|}_{(II)},
        \end{align}
        where the first inequality follows from the triangle inequality.
        We investigate the terms $(I)$ and $(II)$ of~\eqref{eq:g_bound} separately, starting with $(I)$.
\subsubsection{2. Concentration of $(I)$:}
        We assume, without loss of generality, that $\hat{\beta}_t \leq \beta_t$.
        The argument for $\hat{\beta}_t \geq \beta_t$ follows by reversing
        the same argument given below.
        By the Mean Value Theorem and continuity of $g$ with respect to
        its second argument, we have
        \begin{align}
          \label{eq:g_conc}
          &\frac{|g(X,\hat{\beta}_t) - g(X,\beta_t)|}{|\hat{\beta}_t-\beta_t|}
            = \left| \frac{\partial g(X,\beta)}{\partial \beta}
            \Bigg|_{\tilde{\beta} \in [\hat{\beta}_t, \beta_t]}
            \right|
              \leq \max_{\beta \in [\hat{\beta}_t, \beta_t]}
              \left| \frac{\partial g(X,\beta)}{\partial \beta} \right|
              \nonumber
          \\ & \iff
               |g(X, \hat{\beta}_t) - g(X, \beta_t)|
               \leq |\hat{\beta}_t - \beta_t |
               \max_{\beta \in [\hat{\beta}_t, \beta_t]}
               \left| \frac{\partial g(X,\beta)}{\partial \beta} \right|,
        \end{align}
        where
        \begin{align}
          \label{eq:partial_g}
        \frac{\partial g}{\partial \beta}
          = \frac{(\alpha-1)^2 \alpha X^2 e ^{(1-\alpha)X/\beta}}
          {(1+\alpha e ^{(1-\alpha)X / \beta})^2 \beta^2}.
        \end{align}
        Under the constraints that $\beta >0, X > 0$ and $\alpha >1$, we observe
        that the partial derivative~\eqref{eq:partial_g} is always positive. As a
        result, we can drop the absolute value around the partial derivative
        in~\eqref{eq:g_conc} and directly maximize the derivative.
        Now, note that
        \begin{align*}
          (1+\alpha e^{(1-\alpha)X/\beta})^{-1} \in [(1+\alpha)^{-1}, 1],
        \end{align*}
        and that $e^{(1-\alpha)X/\beta} \leq 1$ for $X>0$ since we assume
        $\alpha > 1$.
        Putting this all together gives an upper bound on the maximum of the derivative,
        \begin{align*}
          \max_{\beta \in [\hat{\beta}_t, \beta_t]} \frac{\partial g}{\partial \beta}
          \leq
          \frac{(\alpha - 1)^2\alpha X^2}{[\min\{\hat{\beta}_t, \beta_t\}]^2}.
        \end{align*}
        The $\min\{\hat{\beta}_t, \beta_t\}$ allows for consideration of the
        case $\hat{\beta}_t \geq \beta_t$, and therefore the result holds for
        any $\hat{\beta}_t$. The following analysis is for any $\hat{\beta}_t$.
        By Lemma~\ref{lem:moment_concentration} and using the fact that
        \begin{align*}
          \E[X^2]
          &= \beta^{*2} + \Big(\frac{\beta^*}{\alpha}\Big)^2,
        \end{align*}
        we have that with probability $1-e^{-\eta}$,
        \begin{align*}
          (I)
          &
          \leq
            \frac{(\alpha - 1)^2\alpha|\hat{\beta}_t - \beta_t|}
            {[\min\{\hat{\beta}_t, \beta_t\}]^2}
            \frac{1}{n}\sum_{i=1}^n X_i^2
          \\
          &
            \leq
            \frac{(\alpha - 1)^2\alpha|\hat{\beta}_t - \beta_t|}
            {[\min\{\hat{\beta}_t, \beta_t\}]^2}
            \Big(
            \beta^{*2} + \Big(\frac{\beta^*}{\alpha}\Big)^2
            +
            4C_{\theta}\beta^{*2}\max\{1, \alpha^{-2}\} \sqrt{\frac{\eta}{n}}
            \Big).
        \end{align*}

        \subsubsection{3.  Concentration of $(II)$:}
For term (II) of~\eqref{eq:g_bound}, we invoke sub-Exponential concentration
results which involve computing the sub-Exponential norm of a random variable,
as defined in \cite[Definition 2.8.4]{vershynin2020high}.
The first technical result we need to proceed is the fact that a mixture of
exponentially distributed random variables is sub-Exponential. We capture this,
along with the corresponding sub-Exponential norm in
Lemmas~\ref{lem:subexp-norm} and~\ref{lem:sub_exp-mix}.
Furthermore, it follows from Corollary~\ref{cor:subexp-bond} that $g(X,\beta)$ is a
sub-Exponential random variable with sub-Exponential norm less than or equal to
that of $X$. Thus, by \cite[Theorem
2.9.1]{vershynin2020high}, \cite[Lemma 2.7.8]{vershynin2020high} and
Corollary~\ref{cor:subexp-bond}, we have:
\begin{align}
  \label{eq:prob_rate}
  & \mathbb{P}\left( \left| \frac{1}{n} \sum_{i=1}^n g(X_i, \beta_t)- \mathbb{E}[g(X,\beta_t)] \right|  > \delta \right)
    \nonumber
  \\
  &\leq
    2 \exp \left\{-c \min \left[\frac{n \delta^2}
    {\|g(X_i, \beta_t)- \mathbb{E}[g(X,\beta_t)]\|_{\psi_1}^2},
    \frac{n\delta}{\|g(X_i, \beta_t)- \mathbb{E}[g(X,\beta_t)]\|_{\psi_1}}\right] \right\}
    \nonumber \\
  &
    \leq
    2 \exp \left\{-cn  \min \left[\frac{\delta^2}{B^2 \|X\|_{\psi_1}^2},
    \frac{\delta}{B\|X\|_{\psi_1}}\right] \right\}.
\end{align}
Thus, with probability $1-\delta$, using Lemmas~\ref{lem:subexp-norm} and~\ref{lem:sub_exp-mix},
\begin{align*}
  (II)=
  \left| \frac{1}{n} \sum_{i=1}^n g(X_i, \beta_t)- \mathbb{E}[g(X,\beta_t)] \right|
  \lesssim
  \begin{cases}
    \beta^*\sqrt{\frac{\log(2/ \delta)}{n}},&  \text{if } 0 < \delta \leq B\|X\|_{\psi_1}\\
    \beta^*\frac{\log(2/ \delta)}{n},  &           \text{otherwise}
  \end{cases}
\end{align*}

\subsubsection{4. Recursion:}
Now putting together the bounds for $(I)$ and $(II)$, we have
for any $t>0$,
\begin{align*}
  &
    |\hat{\beta}_{t+1} - \beta_{t+1}|
  \\
  &
    \lesssim
    \frac{(\alpha - 1)^2\alpha \beta^{*2}}
    {[\min\{\hat{\beta}_t, \beta_t\}]^2}
    \Big(
    1 + \frac{1}{\alpha^2}
    +
    4C_{\theta}\max\{1, \alpha^{-2}\} \sqrt{\frac{\eta}{n}}
    \Big)|\hat{\beta}_t - \beta_t|
  \\
  & \qquad
    +
    \beta^* \Big(\sqrt{\frac{\log(2/ \delta)}{n}} + \frac{\log(2/ \delta)}{n}\Big).
\end{align*}
This sets up a recursion in terms of $|\hat{\beta}_{t} - \beta_{t}|$, which we
solve to obtain that at iteration $t$, given an initial estimate $\hat{\beta}_0
= \beta_0$, the absolute loss is bounded like
\begin{align*}
  &
    |\hat{\beta}_{t} - \beta_{t}|
  \\
  &\leq
    \beta^*
    \Big(\sqrt{\frac{\log(2/ \delta)}{n}}
    + \frac{\log(2/ \delta)}{n}\Big)
    +
    O\Big(\frac{1}{\sqrt{n}}
    \sum_{i=1}^{t-1}
    \Big(
    (\alpha - 1)^2\alpha
    ( 1 + \frac{1}{\alpha^2} )
    \Big)^i
    \Big).
\end{align*}
In order to ensure the geometric series contributions remain of smaller order,
we must ensure that we have a sufficiently large sample size, i.e.,
$\sqrt{n} \gtrsim \frac{((\alpha - 1)^2\alpha)^{t-1}}{(\alpha - 1)^2\alpha -1}$,
which must be verified based on the $t$ chosen from Corollary~\ref{cor:pop_conv}
as a function of $\varepsilon$.
\\
\noindent
The conclusion follows by taking $\gamma = \max\{\delta, e^{-\eta}\}$.
\qed

\subsection*{Proof of Corollary \ref{cor:empirical_rate}}
We start by noting that
\begin{align*}
|\hat{\beta}_{t+1}-\beta^*| \leq |\hat{\beta}_{t+1}-\beta_{t+1}| + |\beta_{t+1}-\beta^*|.
\end{align*}
Now, note that, given the assumptions of the theorem, the result follows since
the first term on the RHS above can be bounded using
Theorem~\ref{thm:empirical_bound} and the second term using Corollary~\ref{cor:pop_conv}.
\qed

\section{Additional technical results}
\label{app:technical}
In this section we state and prove technical results that are used in
proving the main statements (see Appendix~\ref{app:main_proofs}) and may be of independent interest.
We start by identifying some key properties of the recursion function of the population EM iterates.
\begin{lem} \label{lem:f_properties}
Let $f(\cdot, \cdot)$ denote the function defined by each step under the
population EM,
\begin{align}
  \label{eq:f}
\beta_{t+1}
  = f(\beta_t, \beta^*)
  = \alpha \mathbb{E}_{\beta^*}[X]
  - (\alpha - 1)\mathbb{E}_{\beta^*}[X \cdot (1+\alpha e^{(1-\alpha)X / \beta_t})^{-1}],
\end{align}
where $\mathbb{E}_{\beta}[\cdot]$ refers to the expectation under the mixture as
defined in~\eqref{eq:mix_density} parameterized by $\beta$.
Then, $f$ satisfies the following properties:
\begin{enumerate}
    \item $f(\beta,\beta) = \beta$
    \item $f(\beta, \mu)$ is increasing with respect to its first argument for
      all $\beta>0$.
\end{enumerate}
\end{lem}

\begin{proof}
  We start with the first property.
  Observe that
  \begin{align*}
    f(\beta, \beta)
    &= \alpha \mathbb{E}_\beta[X] - (\alpha
      -1)\mathbb{E}_\beta[X \cdot (1+\alpha e^{(1-\alpha)X / \beta})^{-1}]
    \\
    &= \frac{(\alpha+1)\beta}{2}-\frac{\alpha -1}{2} \int_0^\infty
x\left(\frac{e^{-x/\beta}}{\beta} + \frac{\alpha e^{-x \alpha
/\beta}}{\beta}\right)(1+\alpha e^{(1-\alpha)x/\beta})^{-1} \diff x
    \\
    &= \frac{(\alpha+1)\beta}{2} - \frac{(\alpha - 1)\beta}{2} = \beta.
  \end{align*}
  Now, for the second property, we start by taking $\beta_1 > \beta_2 >0$.
  Under the assumption that $\alpha > 1$, we note that
  \begin{align*}
    e^{(1-\alpha)x/\beta_1} > e^{(1-\alpha)x/\beta_2}
  \end{align*}
  and thus
  \begin{align*}
    (1+\alpha e^{(1-\alpha)x/\beta_1})^{-1} <
    (1+\alpha e^{(1-\alpha)x/\beta_2})^{-1}.
  \end{align*}
Therefore, for any fixed $\mu$, we have
\begin{align*}
  &f(\beta_1, \mu) - f(\beta_2, \mu)
  \\
  &
    = \frac{\alpha-1}{2}
    \int_0^\infty x\left( \frac{e^{-x/\mu}}{\mu} + \frac{\alpha e^{-x/\mu}}{\mu}\right)
    \left[ \frac{1}{1+\alpha e^{(1-\alpha)x/\beta_2}} - \frac{1}{1+\alpha
    e^{(1-\alpha)x/\beta_1}}\right] \diff x\\
  &\geq 0.
\end{align*}
    Since $\beta_1, \beta_2$ were arbitrarily chosen, the result holds.
    \qed
\end{proof}
The following lemma proves a sufficiently tight lower bound on the partial derivative of
$f$ as defined in~\eqref{eq:f}.
\begin{lem}\label{lem:partial_deriv}
  For any $\mu, \beta>0$ and $\alpha >1$, and $f$ defined as in~\ref{eq:f},
  \begin{align*}
    \frac{\partial f(\beta,\mu)}{\partial \mu} \geq \frac{\alpha+1}{2} -
    \frac{\alpha-1}{2}\left(\frac{3}{e} + \frac{1}{\alpha} \right).
  \end{align*}
\end{lem}

\begin{proof}
    We start first by writing out the full partial derivative of $f$ with
    respect to its second argument.
    \begin{align}
      \label{eq:partial_f}
      \frac{\partial f}{\partial \mu}
      &
      = \frac{\alpha +1}{2}
      \\
      & \
      - \frac{\alpha -1}{2}
      \underbrace{\int_0^\infty x
      \Big[
      \frac{e^{-x/\mu}}{\mu^2} \big(\frac{x}{\mu} - 1\big)
      +\frac{\alpha e^{-x\alpha / \mu}}{\mu^2} \big(\frac{\alpha x}{\mu}-1\big)
      \Big]
      \big(1+\alpha e ^{(1-\alpha)x/\beta}\big)^{-1} \diff x}_{I}.
        \nonumber
    \end{align}
    From here, we partition the integral term, $I$ in~\eqref{eq:partial_f}
    so that each integral can be bound separately:
    \begin{align}
      I &=
          \int_0^\mu
          \frac{e^{-x/\mu}}{\mu^2}\Big( \frac{x}{\mu} - 1\Big)
          (1+\alpha e ^{(1-\alpha)x/\beta})^{-1} \diff x
          \tag{I.1}
          \label{eq:I.1}
      \\
        & \qquad
          + \int_0^{\mu / \alpha}
          \frac{\alpha e^{-x \alpha /\mu}}{\mu^2}\Big( \frac{\alpha x}{\mu} - 1\Big)
          (1+\alpha e ^{(1-\alpha)x/\beta})^{-1} \diff x
          \tag{I.2}
          \label{eq:I.2}
      \\ & \qquad
           +\int_\mu^\infty
           \frac{e^{-x/\mu}}{\mu^2} \Big( \frac{x}{\mu} - 1\Big)
           (1+\alpha e ^{(1-\alpha)x/\beta})^{-1} \diff x
           \tag{I.3}
           \label{eq:I.3}
      \\
        & \qquad
          + \int_{\mu / \alpha}^\infty
          \frac{\alpha e^{-x \alpha /\mu}}{\mu^2}\Big( \frac{\alpha x}{\mu} - 1\Big)
          (1+\alpha e ^{(1-\alpha)x/\beta})^{-1} \diff x.
          \label{eq:I.4}
          \tag{I.4}
    \end{align}
    Note that the integrals~\eqref{eq:I.1} and~\eqref{eq:I.2} are negative
    everywhere and the integrals~\eqref{eq:I.3} and~\eqref{eq:I.4} are positive
    everywhere. Furthermore, for $\beta>0$, note that
    \begin{align*}
      (1+\alpha)^{-1}
      \leq (1 + \alpha e^{(1-\alpha)x / \beta})^{-1}
      \leq 1
      \qquad
      \text{for } x \in [0, \infty).
    \end{align*}
    Thus,
    \begin{align*}
        \eqref{eq:I.3} +\eqref{eq:I.4}
      &\leq \int_\mu^\infty
        \frac{x e^{-x/\mu}}{\mu^2}\Big( \frac{x}{\mu} - 1\Big) \diff x
        + \int_{\mu / \alpha}^\infty \frac{x \alpha e^{-x \alpha /\mu}}{\mu^2}\Big( \frac{\alpha x}{\mu} - 1\Big) \diff x
      \\ &
           =\int_0^\infty \frac{x e^{-x/\mu}}{\mu^2}\Big( \frac{x}{\mu} - 1\Big) \diff x
           + \int_0^\infty \frac{x \alpha e^{-x \alpha /\mu}}{\mu^2}\Big( \frac{\alpha x}{\mu} - 1\Big) \diff x
      \\
      &\qquad
        - \int_0^\mu\frac{x e^{-x/\mu}}{\mu^2}\Big( \frac{x}{\mu} - 1\Big) \diff x
        + \int_0^{\mu /\alpha} \frac{x \alpha e^{-x \alpha /\mu}}{\mu^2}\Big( \frac{\alpha x}{\mu} - 1\Big) \diff x
    \end{align*}
    and
    \begin{align*}
      \eqref{eq:I.1} + \eqref{eq:I.2}
      \leq \frac{1}{1+\alpha} \Big(
      \int_0^\mu\frac{x e^{-x/\mu}}{\mu^2}\Big( \frac{x}{\mu} - 1\Big)\diff x
      +  \int_0^{\mu / \alpha} \frac{x \alpha e^{-x \alpha /\mu}}{\mu^2}\Big( \frac{\alpha x}{\mu} - 1\Big) \diff x \Big).
    \end{align*}
    Therefore,
    \begin{align}
        I& \leq
           \frac{-\alpha}{1+\alpha}
           \Big(
           \int_0^\mu\frac{x e^{-x/\mu}}{\mu^2}\Big( \frac{x}{\mu} - 1\Big)\diff x
           +  \int_0^{\mu / \alpha} \frac{x \alpha e^{-x \alpha /\mu}}{\mu^2}\Big( \frac{\alpha x}{\mu} - 1\Big) \diff x
           \Big)
           + \label{eq:int_by_parts}
      \\ &
           \int_0^\infty\frac{x e^{-x/\mu}}{\mu^2}\Big( \frac{x}{\mu} - 1\Big) \diff x
           +
           \int_0^\infty \frac{x \alpha e^{-x \alpha /\mu}}{\mu^2}\Big( \frac{\alpha x}{\mu} - 1\Big) \diff x.
           \label{eq:moment}
    \end{align}
    We recognize \eqref{eq:moment} as functions of the first moment of each of the
    mixture components, which can be computed directly.
    For~\eqref{eq:int_by_parts}, however, we apply integration by parts to
    compute the integrals.
    This leaves us with the bound
    \begin{align*}
      I
      & \leq
        \frac{-\alpha}{1+\alpha}\Big(1+\frac{1}{\alpha}\Big) \Big( 1- \frac{3}{e}\Big)
        + 2\mu^2\cdot\frac{1}{\mu^2} - \mu\cdot \frac{1}{\mu}
        + \frac{2\mu^2}{\alpha}\cdot\frac{\alpha}{\mu^2}-\frac{\mu}{\alpha}\cdot \frac{1}{\mu}
      \\
      &= \frac{3}{e}+\frac{1}{\alpha}.
    \end{align*}
    Thus, the lower bound on the partial derivative is
    \begin{align*}
      \frac{\partial f}{\partial \mu}
      \geq
      \frac{\alpha +1}{2} -\frac{\alpha -1}{2} \Big( \frac{3}{e} + \frac{1}{\alpha}\Big).
    \end{align*}
    We note that this bound lies in the interval $(0,1)$ for $\alpha \in
    (1,\alpha_{\max})$, where $\alpha_{\max} = \frac{3+\sqrt{9+4(3-e)e}}{2(3-e)}
    \approx 11.49$.
    \qed
  \end{proof}
  \noindent
  The following lemmas and corollary provide bounds on sub-Exponential norms of
  sub-Exponential random variables.
  \begin{lem}\label{lem:subexp-norm}
    Let $X \sim \text{Exp}(\beta)$ be an exponential random variable with scale
    parameter $\beta$.
    Then its sub-Exponential norm is given by
    \begin{align*}
      \|X\|_{\psi_1} = 2\beta.
    \end{align*}
\end{lem}
\begin{proof}
  By the moment generating function of an exponential random variable,
  \begin{align*}
    \mathbb{E}\Big[ e^{X / K}\Big] = \frac{1/\beta}{1/\beta-1/K} = \frac{1}{1-\beta/K}.
  \end{align*}
  Then, by~\cite[Definition 2.8.4]{vershynin2020high}, the sub-Exponential norm
  is defined as
  \begin{align*}
    \|X\|_{\psi_1} = \inf \{K : \mathbb{E}[e^{|X| / K}] \leq 2 \},
  \end{align*}
  and as such,
  \begin{align*}
    \mathbb{E}\Big[ e^{|X| / K}\Big] \leq 2
    \iff
    K \geq 2\beta
  \end{align*}
  and the result follows.
  \qed
\end{proof}

\begin{lem} \label{lem:mono}
Let $X, Y $ be non-negative random variables such that $X \leq Y$ almost
everywhere. Then, the sub-Exponential norm $\| \cdot \|_{\psi_1}$ as defined in
\cite[Definition 2.8.4] {vershynin2020high} satisfies the following:
\begin{align*}
 \|X\|_{\psi_1} \leq \|Y\|_{\psi_1}.
\end{align*}
In other words, the sub-Exponential norm is monotonically increasing.
\end{lem}
\begin{proof}
    As $0 \leq X \leq Y$ a.e., we have
    \begin{align}\label{subexp-norm:ineq}
    \mathbb{E}\left[e^{|X| /K}\right] \leq \mathbb{E}\left[e^{|Y| /K}\right],
    \end{align}
    for all $K >0$. Let $K_Y = \inf \{ K : \mathbb{E}\left[e^{|Y| /K} \right] \leq
    2\}$.
    Then by (\ref{subexp-norm:ineq}), we clearly have that
    \begin{align*}
      \inf \{ K : \mathbb{E}\left[e^{|X| /K} \right] \leq 2\} \leq K_Y,
    \end{align*}
     that is, $K_X = \inf \{ K : \mathbb{E}\left[e^{|X| /K} \right] \leq 2\} \leq
     K_Y$.
     But $K_X,K_Y$ are exactly the sub-Exponential norms of $X,Y$ respectively,
     and so we have shown the claim.
     \qed
\end{proof}

\begin{cor}\label{cor:subexp-bond}
    Let $X$ be a non-negative random variable and let $h(X)$ be some
    non-negative function of $X$ that is bounded almost everywhere, that is, $h(X)
    \leq C$ a.e., for some constant $C>0$. Then,
    \begin{align*}
      \|Xh(X)\|_{\psi_1} \leq C\|X\|_{\psi_1}.
    \end{align*}
  \end{cor}
  \begin{proof}
The corollary follows directly from Lemma~\ref{lem:mono}, as $CX \geq X h(X)$ almost everywhere.

\end{proof}
\begin{lem}
    \label{lem:sub_exp-mix}
    Let $X\sim \sum_{k=1}^n \pi_k X_k$ be a random variable generate from a
mixture model where each component $X_k$ is a sub-Exponential random variable.
Then, $X$ is also a sub-Exponential random variable with sub-Exponential norm
\begin{align*}
  \|X\|_{\psi_1} \leq \max_k \|X_k\|_{\psi_1}.
\end{align*}
\end{lem}
\begin{proof}
  In order to show that $X$ is sub-Exponential, it suffices to prove that there
  exists some $K>0$ such that $\mathbb{E}\left[e^{ |X| /K}\right] \leq 2$.
  First, by definition of $X$ being a mixture,
  \begin{align*}
    \mathbb{E}\left[e^{ |X| /K}\right]
    &= \int_{- \infty}^\infty e^{ |x| /K} \sum_{k=1}^n \pi_k p_k(x) \diff x
      =\sum_{k=1}^n \pi_k \int_{- \infty}^\infty e^{ |x| /K} p_k(x) \diff x
    \\
    &= \sum_{k=1}^n \pi_k\mathbb{E}\left[e^{ |X_k| /K}\right],
  \end{align*}
  where $p_k(\cdot)$ is the density of the $k$th component of the mixture model.
  Now, setting $K = \max_k \|X_k\|_{\psi_1}$, we have
  \begin{align*}
    \sum_{k=1}^n \pi_k\mathbb{E}\left[e^{ |X_k| /K} \right]
    \leq 2\cdot\sum_{k=1}^n
    \pi_k = 2,
  \end{align*}
  where the second inequality holds by the fact that $e^{x/s} \leq e^{x/t}$ for
  any $0 < s \leq t$. It thus follows by the definition of the sub-Exponential norm that $\|X\|_{\psi_1} \leq \max_k \|X_k\|_{\psi_1}.$ \qed
\end{proof}

\begin{lem}[Second moment concentration]
  \label{lem:moment_concentration}
  Let $X \sim \frac{1}{2}\text{Exp}(\beta)+\frac{1}{2}\text{Exp}(\beta/\alpha)$.
  Then, with probability $1- e^{-\delta}$
  \begin{align*}
    \left| \frac{1}{n} \sum_{i=1}^n X_i^2 - \E[X^2]\right|
    \leq
    4 C_\theta \beta^{2}\max\{1, \alpha^{-2}\} \sqrt{\frac{\delta}{n}},
  \end{align*}
  where $C_{\theta}$ depends only on the tail index $\theta$ of the Weibull
  parametrization of $X^2$.
\end{lem}
\begin{proof}
  We first note that by Lemmas ~\ref{lem:sub_exp-mix} and~\ref{lem:subexp-norm}
  each $X_i$ is sub-Exponential with sub-Exponential norm $\|X_i\|_{\psi_1} =
  2\max\{\beta, \frac{\beta}{\alpha}\}$.
  In order to show concentration of $X^2$, we need to define the sub-Weibull property.
  We use the definition as provided in \cite[Corollary
  6.1]{zhang2020concentration} which states that a random variable $Y$ is said
  to be sub-Weibull with tail index $\theta>0$ if
  \begin{align*}
    \mathbb{P}(Y\geq y) \leq a \exp(-by^{\theta}), \ \forall\ y\geq 0,
  \end{align*}
  where $a, b>0$ are some fixed constants.
  Note that by this definition, each $X^2_i$ is sub-Weibull with tail
  index $\theta = \frac{1}{2}$ since
  \begin{align*}
    \mathbb{P}(X^2 \geq x)
    = \mathbb{P}(X \geq \sqrt{x})
    = 2\exp\Big(-\frac{\sqrt{x}}{\beta}\Big).
  \end{align*}
  Now define the sub-Weibull norm~\cite[Definition
  6.2]{zhang2020concentration} for each $X_i^2$ as
  \begin{align*}
    \|Y\|_{\psi_\theta} = \inf \{C \in (0, \infty): \E[e^{|Y|^\theta/C^{\theta}}]\leq 2\}.
  \end{align*}
  We now compute the sub-Weibull norm for the square of each component of
  the mixture. Let $X_{\beta}$ be the random variable corresponding to the
  mixture component with shape parameter $\beta$. $X_{\beta/\alpha}$ is
  defined identically. The sub-Weibull norm of $X_{\beta}$ with tail index
  $\theta=\frac{1}{2}$ is given by
  \begin{align*}
    \|X_{\beta}^2\|_{\psi_{1/2}}
    =
    \inf \Big\{C \in (0, \infty):
    \E\Big[e^{\frac{(X_{\beta}^2)^{1/2}}{C^{1/2}}}\Big] \leq 2
    \Big\}.
  \end{align*}
  Thus, using the moment generating function of an exponential random
  variable with shape parameter $\beta$,
  \begin{align*}
    \frac{1}{1-\frac{\beta}{\sqrt{C}}}\leq 2
    \iff
    C \geq 4\beta^2.
  \end{align*}
  Using the fact that the sub-Weibull norm of a mixture is bounded by
  the maximum sub-Weibull norm over all components,
  \begin{align*}
    \|X\|_{\psi_{1/2}}
    \leq
    \max\{\|X_{\beta}^2\|_{\psi_{2}}, \|X_{\beta/\alpha}^2\|_{\psi_{2}}\}
    = 4 \beta^{2}\max\{1, \alpha^{-2}\}.
  \end{align*}
  Now, applying the concentration bound from~\cite[Proposition
  3]{zhang2022sharper}, with probability $1- e^{-\delta}$
  \begin{align*}
    \left| \frac{1}{n} \sum_{i=1}^n X_i^2 - \E[X^2]\right|
    \lesssim
    4 C_\theta \beta^{2}\max\{1, \alpha^{-2}\} \sqrt{\frac{\delta}{n}},
  \end{align*}
  where $C_{\theta}$ denotes an absolute constant that depends only on the tail
  index $\theta$.
  \qed
\end{proof}

\end{subappendices}

\end{document}